\documentclass[reqno,12pt]{amsart}
\usepackage{amsmath,amssymb,amsthm,graphicx,mathrsfs,url}
\usepackage[usenames,dvipsnames]{color}
\usepackage[colorlinks=true,linkcolor=Red,citecolor=Green]{hyperref}
\usepackage{amsxtra}
\usepackage{subcaption}
\usepackage{float}
\usepackage{hyperref}

\def\arXiv#1{\href{http://arxiv.org/abs/#1}{arXiv:#1}}

\def\?[#1]{\textbf{[#1]}\marginpar{\Large{\textbf{??}}}}
\def\smallsection#1{\smallskip\noindent\textbf{#1}.}
\let\epsilon=\varepsilon 

\newcommand{\RR}{{\mathbb R}}

\newcommand{\E}{{\mathbb E}}

\newtheorem{theo}{Theorem}
\newtheorem{prop}{Proposition}[section]	
\newtheorem{defi}[prop]{Definition}

\newtheorem{ass}{Assumption}
\newtheorem{ex}{Example}
\newtheorem{lemm}[prop]{Lemma}
\newtheorem{corr}[prop]{Corollary}
\newtheorem{rem}{Remark}
\numberwithin{equation}{section}

\let\Re=\Real

\DeclareMathOperator{\tr}{tr}

\def\indic{\operatorname{1\hskip-2.75pt\relax l}}
\usepackage{scalerel}

\newcommand\reallywidehat[1]{\arraycolsep=0pt\relax%
\begin{array}{c}
\stretchto{
  \scaleto{
    \scalerel*[\widthof{\ensuremath{#1}}]{\kern-.5pt\bigwedge\kern-.5pt}
    {\rule[-\textheight/2]{1ex}{\textheight}} 
  }{\textheight} %
}{0.5ex}\\           
#1\\                 
\rule{-1ex}{0ex}
\end{array}
}

\title[Model reduction of controlled stochastic dynamics]{Error bounds for model reduction of feedback-controlled linear stochastic dynamics on Hilbert spaces}
\author{Simon Becker}
\email{simon.becker@damtp.cam.ac.uk}
\address{DAMTP, University of Cambridge, Wilberforce Road, Cambridge CB3 0WA, UK}
\author{Carsten Hartmann}
\email{carsten.hartmann@b-tu.de}
\address{Institute for Mathematics, Brandenburgische Technische Universit\"at, Konrad-Wachsmann-Allee 1, 03046 Cottbus, Germany}
\author{Martin Redmann}
\email{martin.redmann@mathematik.uni-halle.de}
\address{Martin Luther University Halle-Wittenberg, Institute for Mathematics,
Theodor-Lieser-Str.  5, D-06120 Halle (Saale), Germany}
\author{Lorenz Richter}
\email{lorenz.richter@gmail.com}
\address{Institute of Mathematics, Freie Universit\"at Berlin, Berlin; Institute for Mathematics, Brandenburgische Technische Universit\"at, Konrad-Wachsmann-Allee 1, 03046 Cottbus, Germany}

\begin{document}

\begin{abstract}
We analyze structure-preserving model order reduction methods for Ornstein-Uhlenbeck processes and linear S(P)DEs with multiplicative noise based on balanced truncation. {For the first time, we include in this study the analysis of non-zero initial conditions. We moreover allow for feedback-controlled dynamics for solving stochastic optimal control problems with reduced-order models and prove novel error bounds for a class of linear quadratic regulator problems. We provide numerical evidence for the bounds and discuss the application of our approach to enhanced sampling methods from non-equilibrium statistical mechanics.}
\end{abstract}

\maketitle


\section{Introduction}
{
In this paper, we consider optimal control problems for Ornstein-Uhlenbeck processes and linear stochastic systems with multiplicative noise in a Hilbert space setting. These (abstract) equations describe stochastic partial differential equations (SPDEs) or high dimensional stochastic differential equations (SDEs) representing spatial discretizations of underlying SPDEs. Since optimal control problems in large (or infinite) dimensions often require high computational effort, thereby rendering practical applications infeasible, we resort to model order reduction (MOR) techniques. Here, the key idea is to identify low-order approximations capturing the dynamics of the originally large-scale systems such that, subsequently, the optimal control problem can be solved in the dimensionally reduced setting in which the complexity is lower or in which algorithms can be applied that would not be feasible in the original framework.}

{Even though MOR of linear and bilinear control systems is often justified by the incentive to reduce the computational burden associated with solving optimal control problems, they are usually not designed for approximating feedback-control problems. Most of the standard techniques like Gramian-based (balanced) MOR \cite{CG,BTsurvey}, proper orthogonal decomposition \cite{POD,Rowley}, or interpolation-based MOR \cite{SdS,H2survey} are \emph{open-loop} methods. Therefore, most of the error analysis focuses on worst-case error bounds (if any) for a certain class of (e.g. square-integrable) admissible controls. The specifics of the control or the cost functional are not taken into account for the identification of the relevant subspace or the bounds on the approximation error, which explains that the typical error bounds, such as the Hankel norm or $L^\infty$-error bounds of balanced truncation, are rather conservative when applied to feedback control, i.e.~\emph{closed-loop} systems. Notable exceptions are linear quadratic Gaussian control (LQG) balancing methods \cite{LQGBT} and some of their more recent variants (e.g. \cite{LQGPS,LQGPHS}) that are based on balancing a pair of control and filter Riccati equations associated with an underlying linear quadratic regulator problem; other approaches include stochastic (Feynman-Kac or backward SDE) representations of the associated Hamilton-Jacobi-Bellman (HJB) equation \cite{HLZ,KNH18}.} 

{In this paper, we follow an alternative route to LGQ balancing or dynamic programming using stochastic representations of HJB equations and instead consider Gramian-based MOR with the goal to tighten the available open-loop error bounds. The motivation for this strategy is that the computational cost associated with solving Lyapunov equations for the Gramians rather than algebraic Riccati equations or HJB equations (e.g. using monotone finite difference schemes \cite{BS91} or deep learning \cite{NR21}) is greatly reduced. We should mention that for Ornstein-Uhlenbeck type systems with quadratic cost functionals, it is possible to reduce the corresponding HJB equations to Riccati equations, which further reduces the computational overhead of grid-based discretization schemes for HJB equations. Nevertheless, Lyapunov equations in infinite dimensions are relatively well-behaved, which cannot be said for the corresponding operator Riccati equations (e.g.~see \cite{OC05}), thereby further motivating our study.}

To fix ideas, let $(M_t)_{t\ge 0}$ be a square-integrable mean zero L\'evy process and let $(\mathcal F_t)_{t \ge 0}$ be its induced filtration. For control functions $u \in L^2_{\operatorname{ad}}(\Omega \times (0,T))$ with values in $\mathbb{R}^{m}$, we study the differential equations
\begin{subequations}
\label{eq:standardform}
\begin{eqnarray}
\label{eq:OU}
&\mathrm dZ^{\operatorname{ou}}_t = A Z^{\operatorname{ou}}_t \ \mathrm dt +Bu_t \ \mathrm dt+  K \ \mathrm dM_t\,, \quad Z^{\operatorname{ou}}_{t_0}=\xi\\
\label{eq:bil}
&\mathrm dZ^{\operatorname{lin}}_t = A Z^{\operatorname{lin}}_t \ \mathrm dt + Bu_t \ \mathrm dt+  N  Z^{\operatorname{lin}}_t  \ \mathrm dM_t\,,  \quad Z^{\operatorname{lin}}_{t_0}=\xi 
\end{eqnarray}
\end{subequations}
{for $t\in (t_0,T)$ on some separable Hilbert space $X$. We will mostly consider the case $t_0=0$, and throughout the paper we use the labels ``ou'' for  Ornstein-Uhlenbeck processes and ``lin'' for linear systems with multiplicative noise that are sometimes also referred to as ``bilinear'' in the literature  (e.g. \cite{BD11}).} In \eqref{eq:OU} the process is allowed to take values in $\mathbb{R}^d$, whereas in equation \eqref{eq:bil} the L\'evy process is assumed to be scalar\footnote{This assumption is only to simplify the notation in this article and an adaptation to multiple noise terms $\sum_{i=1}^{l} N_i  Z^{\operatorname{lin}}_t  \ \mathrm dM^{i}_t$ is straightforward.}. The precise assumptions we impose on the OU process \eqref{eq:OU} are stated in Section \ref{sec:stoch} and for equation \ref{ass:lin} in Section \ref{sec:bilstoch}. Most of the notation will be explained in Section \ref{sec:Not}. In the equations above $B: \mathbb R^m \rightarrow X$, is the linear \emph{input operator}. In general, we are interested in outputs $CZ_{t}$, where $C:X \rightarrow \mathcal H$ is the linear \emph{output operator}.

The study of controlled Ornstein-Uhlenbeck processes \eqref{eq:OU} is of great practical relevance and has various applications such as interest rates models \cite{V} or pair trading in mathematical finance \cite{ES}, and Langevin equations in physics \cite{K}. Such processes are also considered to model random perturbations of linear deterministic systems \cite{HNS}. 
Linear stochastic differential equations with multiplicative noise generalizes a dissipative geometric Brownian motion and has multiple applications in mathematical finance, where, most prominently, such equations describe stock prices in the Black-Scholes model \cite{H}. Examples involving SPDEs include stochastic variants of the linearized Navier-Stokes equations \cite{Donev2014}, stochastic polymer models \cite{Medina1989}, or the Kushner-Stratonovich equations of nonlinear filtering \cite{Bucy1965}. 

\subsection{Optimal control}\label{sec_opt_control}

For the optimal control problem associated to the equations in \eqref{eq:standardform}, we consider quadratic cost functions on a time horizon $T \in (0,\infty)$, given by
\begin{equation}
\begin{split}
\label{eq:energy2}
J_{\operatorname{LQR}}^{\text{ou}} (CZ^{\text{ou}},u,T)&:=  \frac{1}{T}\left(\left\lVert C Z^{\text{ou}} \right\rVert^2_{L^2(\Omega_T)}+ \langle u, R u \rangle_{L^2(\Omega_T)}\right)\\
J_{r}^{\text{lin}} (CZ^{\text{lin}},u,T)&:=  \left\lVert C Z^{\text{lin}} \right\rVert^2_{L^2_tL^r_{\omega}(\Omega_T)}+ \langle u, R u \rangle_{L^2(\Omega_T)} \text{ with }r \in (1,2].
\end{split}
\end{equation} 
(For the definition of the corresponding norms and scalar products, see Section \ref{sec:Not} below.)
{We consider only functionals of quadratic type, as they allow us to use an explicit representation of the optimal feedback control using LGQ theory, which is necessary to obtain our error bounds. Let us remark that the functionals in \eqref{eq:energy2} are defined slightly differently compared to some of the control applications appearing in the literature in order to acknowledge the fact that a stable OU dynamics with uniformly bounded second moment is not decaying, in contrast to a (mean-square) stable dynamics with multiplicative noise. For a fixed simulation time $T$, the regularization by $1/T$ of the first control functional can be omitted, however, it becomes necessary for an infinite simulation time. To be precise,} in case of $T=\infty$ we define for the Ornstein-Uhlenbeck process 
\[J_{\operatorname{LQR}}^{\text{ou}} (CZ^{\text{ou}},u,\infty) :=\limsup_{T \rightarrow \infty}J_{\operatorname{LQR}}^{\text{ou}} (CZ^{\text{ou}},u,T).  \]
In \eqref{eq:energy2}, $R$ is a (strictly) symmetric positive-definite matrix such that all eigenvalues of $R$ are strictly positive. 
{For the ease of notation, we suppress the explicit dependence of the cost in case of a finite time horizon, $T<\infty$, on the initial data $(t_0,\xi)$. In the case $T=\infty$ and under some suitable ergodicity (i.e.~stability and complete controllability) assumptions, the optimal cost after taking the infimum over the controls $u$ can be shown to be independent of the initial conditions \cite{BorkarBook}.}

\subsection{Model order reduction}
{
As mentioned above, MOR shall be applied in order to lower the complexity of the problem discussed in Section \ref{sec_opt_control}. We mainly focus on a Gramian-based approach called balanced truncation (BT). Gramians can be interpreted as algebraic structures that are constructed to identify less relevant directions in state equations such as \eqref{eq:OU} and \eqref{eq:bil} as well as redundant information in the quantity of interest $C Z_t$. Simultaneous diagonalization of these Gramians then allows to easily detect and truncate unimportant states in order to find an accurate reduced system.}

{It turns out that MOR of control systems is intimately related with MOR of non-zero initial conditions. Therefore a few remarks on the specifics of Gramian-based BT in connection with non-zero initial conditions are in order.}
\subsubsection{Deterministic systems}
BT is very popular in the context of linear and bilinear deterministic control systems, since it features computable error bounds and preserves many structural properties of the dynamics, such as stability or passivity (e.g.~\cite{SchildersBook}). Nevertheless, considered as an approximation tool for the Hankel operator that is underlying the system under consideration, it heavily relies on $L^{2}$-isometries and the fact that inputs and outputs are square-integrable functions on the positive reals \cite{Glover84}. With few exceptions (see~\cite{BGM,HRA,DHQ}), most of the available error bounds consider the dynamics under zero (or: homogeneous) initial conditions.  
This is somewhat surprising as, for example, the system-theoretic concepts of finite-time controllability and reachability make assertions about bounded measurable control inputs only and do not assume the initial condition to be zero (see, e.g. \cite[Sec.~4]{Casti85}). It is  possible to think of the initial conditions as an extra control input, however, the control input associated with the initial condition is a Dirac delta function, and as a consequence it is neither bounded nor square-integrable; the approach thus requires an appropriate regularization that then leads to Hankel norm error bounds that depend on the particular regularization chosen (see e.g. \cite{HRA}).

\subsubsection{Stochastic setting}
In this article, we follow a different route and extend the notion of the Hankel operator to account for the non-zero initial conditions by an appropriate shifting of the underlying reachability and observability Gramians. The details will be given below in Section \ref{sec:BTIAN}. {In doing so, we study balanced MOR methods for \eqref{eq:standardform} under non-zero initial states.} Reduced order models, based on BT, for (uncontrolled) Ornstein-Uhlenbeck processes $(Z^{\operatorname{ou}}_t)$ have been considered in \cite{FR2}; controlled processes $(Z^{\operatorname{lin}}_t)$ have been extensively studied within the standard stochastic BT framework and we refer the reader to \cite{BH19,BR15,BD11} and references therein for a general overview. {To our knowledge, non-zero initial conditions for equations like  \eqref{eq:bil} have in general not been considered in the BT MOR framework so far.}

{For stochastic control problems, for which the optimal policies are know to be Markovian feedback controls, the dependence of the controlled dynamics on the initial conditions is crucial \cite[Sec. III.7]{FlemingSoner}. In contrast to the deterministic case, dynamic programming, i.e. (approximately) solving HJB equations, or stochastic optimization methods are the methods of choice to compute optimal controls, and these methods rely on a careful treatment of the initial data. For example, the solution to the HJB equation, the value function, is a function of the initial conditions, and the optimal control can often be expressed in terms of the derivatives of the value function.}
%
%
{As we will detail below, we include the initial states in the MOR process by projecting them on an $L^2$-subspace that is spanned by the \emph{admissible initial states}, which guarantees that we can treat control and initial data on the same footing.} 


\subsubsection{Differences between \eqref{eq:OU} and \eqref{eq:bil}}
The treatment of OU processes and systems with multiplicative noise seems analogous and follows a similar guiding principle in our work, but has fundamental differences. While we consider the same noise processes $(M_t)$ for both equations, the assumptions on the considered dynamics are different. The assumption on the OU-type dynamics requires a strictly dissipative linear part, whereas we require {a slightly stronger stability condition} for the stochastic dynamics for systems with multiplicative noise. In case of OU processes, we work directly with the underlying semigroup, whereas for systems with multiplicative noise, it is the stochastic flow generated by the uncontrolled part, that takes on the fundamental position.
{What prevents us from putting the two dynamics \eqref{eq:OU} and \eqref{eq:bil} under the same umbrella are the different mathematical structures of the two equations, which force us to use different estimates. Specifically, we end up controlling different norms of the solution, even though we enforce the same square integrability condition on the controls. This is unavoidable, and it is owed to the fact that for the OU process with additive noise large randomness will induce a large norm, whereas for systems with multiplicative noise, the effect of the noise on the norm of the solution depends by the magnitude of the process itself.}

\subsection{Outline}

The rest of the article is organized as follows: Before presenting BT in a nutshell in Section \ref{sec:BTIAN}, Section \ref{sec:Not} briefly introduces the basic notation for this article. The OU semigroup and the corresponding model reduction error bound are discussed in Section \ref{sec:stoch}, whereas linear S(P)DEs with multiplicative noise are the subject of Section \ref{sec:bilstoch}. The OU and S(P)DE error bounds are then revisited from the perspective of optimal control theory in Section \ref{sec:OCT}, where we focus on linear quadratic regulator (LQR) problems.  Finally, in Section \ref{sec:num}, we illustrate the theoretical findings from Sections  \ref{sec:stoch}--\ref{sec:OCT} with suitable numerical examples.  

\subsection{Further notation}
\label{sec:Not}
The space of bounded linear operators between Banach spaces $X,Y$ is denoted by $\mathcal L(X,Y)$ and just by $\mathcal L(X)$ if $X=Y.$ The operator norm of a bounded operator $T \in \mathcal L(X,Y)$ is written as $\left\lVert T \right\rVert$. The trace-class and Hilbert-Schmidt operators between Hilbert spaces $X,Y$ are denoted by $\operatorname{TC}(X,Y)$ and $\operatorname{HS}(X,Y),$ respectively. In particular, we recall that for a linear operator $T \in \operatorname{TC}(X,Y)$, where $X$ and $Y$ are now separable Hilbert spaces, the trace norm is given as
\begin{equation}
\label{tracenorm}
\left\lVert T \right\rVert_{\operatorname{TC}}=\sup \left\{ \sum_{n \in \mathbb{N}} \left\lvert \langle f_n,T e_n  \rangle_Y \right\rvert : (e_n)_{n \in \mathbb{N}} \text{ ONB of } X \text{ and } (f_n)_{n \in \mathbb{N}} \text{ ONB of } Y \right\}.
\end{equation}
The Hilbert-Schmidt norm is given by
\begin{equation}
\label{hsnorm}
\left\lVert T \right\rVert_{\operatorname{HS}}:= \sqrt{\sum_{n,m \in \mathbb{N}} \left\lvert \langle f_m,T e_n  \rangle_Y \right\rvert^2} 
\end{equation}
where $(e_n)_{n \in \mathbb{N}}$ is any ONB of $X$ and $(f_n)_{n \in \mathbb{N}}$ any ONB of $Y.$

We say that $g=\mathcal O (f)$ if there is a $C>0$ such that $\left\lVert g \right\rVert \le C \left\lVert f \right\rVert.$ 
The domain of unbounded operators $A$ is denoted by $D(A).$

We write $\Delta(\Xi)$ to denote the difference of the quantity $\Xi$ for two systems, i.e.\@ $\Delta(\Xi) = \Xi_{\operatorname{System \ 1}}-\Xi_{\operatorname{System \ 2}}.$ We denote the expectation of a random variable $Y$ by $\mathbb E(Y)$ where we throughout the article assume to work on some fixed probability space $(\Omega,\mathcal F,\mathbb P).$ If we want to address an operator $L$ for both OU processes and linear systems with multiplicative noise, we write $L^{\operatorname{ou}\vert \operatorname{lin}}.$ 

We write $\Omega_T:=\Omega \times (0,T)$ and define, for a Banach space $Y$, the norm associated with the space $L^2(\Omega_T,Y)$
\begin{equation}
\label{eq:L2space}
\left\lVert f \right\rVert_{L^2(\Omega_T,Y)}:= \sqrt{\mathbb E \int_{(0,T)} \Vert f(t) \Vert^2_Y \, \mathrm dt }.
\end{equation}

When writing $L^p$ spaces, we most often omit the domain and sometimes also the image space to shorten the notation. 

We also define the norm on iterated $L^pL^q$ spaces by
\begin{equation}
\label{eq:lplq}
\Vert f \Vert_{L^p_xL^q_y}:= \Vert x \mapsto  \Vert y \mapsto f(x,y) \Vert_{L^q} \Vert_{L^p}\,,
\end{equation}
where the $L^{q}$ norm is taken over the second argument, $y$, followed by the $L^{p}$ norm integration over the first argument, $x$.  

We use the subscript $ad$ for $L^p$ spaces to denote stochastic processes in $L^p$ that are adapted to a canonical filtration.

The convolution of two functions is denoted by 
\[ (f*g)(x) = \int_{\mathbb R} f(x-y) g(y) \ \mathrm dy. \]

We write $\indic_X$ for the indicator function on some measurable set $X$, i.e. $\indic_X(x)=1$ if $x \in X$ and $0$ otherwise. 

If a sequence $(x_n)$ converges with respect to the weak topology of a Banach space to some element $x$ of that space, we write $x_n \rightharpoonup x.$

To include subspaces of relevant initial states in the MOR process, we define for an orthonormal family $\phi_i\in L^2(\Omega,X)$, the map $B_{\operatorname{in}}: \mathbb R^k \rightarrow X$ by $B_{\operatorname{in}}v:=\sum_{i=1}^k \left\langle v, \widehat{e_i} \right\rangle_{\mathbb R^k} \ \phi_i.$ Here,  $\operatorname{span}\left\{\phi_i; i \in \left\{1,..,k \right\}\right\}$ is the space of \emph{admissible initial states}. In other words, we define an operator $B_{\operatorname{in}}$ such that $B_{\operatorname{in}}B_{\operatorname{in}}^{*}$ is a projection onto the subspace of admissible initial states.

\section{Balanced truncation in a nutshell}
\label{sec:BTIAN}
In this article, we study MOR methods for equations \eqref{eq:standardform}. To fix ideas, let us for now assume that the underlying Hilbert space $X$ is finite-dimensional. In the first step of the MOR process, positive semidefinite observability and reachability  Gramians $\mathscr O^{\operatorname{ou}\vert \operatorname{lin}}$ and $\mathscr P^{\operatorname{ou}\vert \operatorname{lin}}$ are computed from Lyapunov equations, using an auxiliary operator $S:=B_{\operatorname{in}}B_{\operatorname{in}}^{*}+ K \mathbb E(M_1 M_1^{*}) K^{*}$. {We note that Gramians are the key ingredient of balancing-related MOR methods like BT, since from their eigenspaces, dominant subspaces of the underlying system can be extracted.}

For Ornstein-Uhlenbeck processes, for which we consider two types of reachability Gramians $ \mathscr P^{\operatorname{ou}}$ and $ \mathcal P^{\operatorname{ou}}$, the Lyapunov equations (Prop. \ref{eq:LyapOU}) take the form
\begin{equation}
\begin{split}
\label{eq:Lyapunoveq}
&A^{*}\mathscr O^{\operatorname{ou}} + \mathscr O^{\operatorname{ou}} A  + C^{*}C = 0,  \\
&A\mathscr P^{\operatorname{ou}}+ \mathscr P^{\operatorname{ou}}A^{*}+ BB^{*}+K \mathbb E(M_1 M_1^{*}) K^{*}=0 \text{ and }\\
&A( \mathcal P^{\operatorname{ou}}-S)+ ( \mathcal P^{\operatorname{ou}}-S)A^{*}+ BB^{*}=0.
\end{split}
\end{equation}

\begin{itemize}
\item The first reachability Gramian $\mathscr P^{\operatorname{ou}}$ is employed to obtain an error bound on the supremum norm with initial state $0$ (Theorem \ref{general_bound_same_u}), {which is the basis for a bound with general initial states relying on the same type of Gramian (Corollary \ref{L2-error-nonzero-initial-split})}.  

\item The second reachability Gramian $\mathcal P^{\operatorname{ou}}$ depends also on the chosen initial states and allows us to obtain an $L^2$ error bound (Theorem \ref{theo:stoch}). 
\end{itemize}

For linear systems with multiplicative noise they satisfy (see Prop. \ref{prop:Lyapbil} below)
\begin{equation}
\begin{split}
\label{eq:Lyapunoveq2}
&A(\mathscr P^{\operatorname{lin}}-B_{\operatorname{in}}B_{\operatorname{in}}^{*}) +(\mathscr P^{\operatorname{lin}}-B_{\operatorname{in}}B_{\operatorname{in}}^{*})A^{*} +  N (\mathscr P^{\operatorname{lin}}-B_{\operatorname{in}}B_{\operatorname{in}}^{*}) N^{*} \\
&\qquad + BB^{*}+B_{\operatorname{in}}B_{\operatorname{in}}^{*}=0 \text{ and }\\
&A^{*}\mathscr O^{\operatorname{lin}} + \mathscr O^{\operatorname{lin}} A + N^{*} \mathscr O^{\operatorname{lin}} N + C^{*}C = 0.
\end{split}
\end{equation}
Since both $\mathscr O^{\operatorname{ou}\vert \operatorname{lin}}$ and $\mathscr P^{\operatorname{ou}\vert \operatorname{lin}}$ are positive semidefinite, they can be decomposed as $\mathscr O^{\operatorname{ou}\vert \operatorname{lin}} = W^{*}W$ and $\mathscr P^{\operatorname{ou}\vert \operatorname{lin}} =RR^{*}.$
Let $\mathscr O^{\operatorname{ou}\vert \operatorname{lin}}$ and $\mathscr P^{\operatorname{ou}\vert \operatorname{lin}}$ have for simplicity full rank, the \emph{balanced representation} is obtained by first performing a singular value decomposition $WR = V \Sigma U^{*}$, to identify a dominant subspace for the dynamics of the system, where $V,U$ are unitary and $\Sigma$ is diagonal. {The diagonal entries of $\Sigma$ are called Hankel singular values of the system}. Then, we conjugate the system by operators $T:=\Sigma^{-1/2} V^{*}W$ and $T^{-1}:=RU\Sigma^{-1/2}$ such that 
\begin{equation}\label{blanced_realization}
\begin{split}
& A_b^{\operatorname{ou}\vert \operatorname{lin}} := TA^{\operatorname{ou}\vert \operatorname{lin}}T^{-1},\ B_b^{\operatorname{ou}\vert \operatorname{lin}}:=TB^{\operatorname{ou}\vert \operatorname{lin}}, \, C_b^{\operatorname{ou}\vert \operatorname{lin}}:=C^{\operatorname{ou}\vert \operatorname{lin}}T^{-1}, \text{ and }\\
& N_b^{\operatorname{lin}}:=TN^{\operatorname{lin}}T^{-1}, \  K_b^{\operatorname{ou}}:=TK^{\operatorname{ou}}.
\end{split}
\end{equation}
{The state space transformation in \eqref{blanced_realization} can be viewed as a procedure in which the Gramians $\mathscr O^{\operatorname{ou}\vert \operatorname{lin}}$ and $\mathscr P^{\operatorname{ou}\vert \operatorname{lin}}$ are simultaneously diagonalized. This is done because the Hankel singular values characterize the importance of associated state components as shown in other stochastic settings for additive noise \cite{FR2} and multiplicative noise \cite{redmannspa2}.
}
To obtain a \emph{reduced system} {by BT}, the operator $\Sigma$ is now approximated. This approximation is obtained by discarding the smallest singular values of $\Sigma$ { and only capturing the large ones since the corresponding state variables mainly determine the dynamics. }
Error bounds in this article are commonly expressed in terms of the difference of Hankel operators for the full and the reduced system. This difference of Hankel operators we denote by $\Delta(H).$ The Hankel operator is one possible decomposition $WR$ of the Gramians above. The precise definition of the Hankel operator is stated in Definitions \ref{defi:OUHankel}, for OU processes, and \ref{defi:bilHank}, for linear systems with multiplicative noise, respectively. However, to evaluate the trace norm difference it is not necessary to analyze the Hankel operator directly:
{
To evaluate the singular values of $\Delta(H)$, and thus the trace norm of $\Delta(H)$, we introduce an \emph{error system} 
\begin{equation}
\begin{split}
\label{eq:compsys}
&\widehat A^{\operatorname{ou}\vert \operatorname{lin}}:=\left(\begin{matrix} A & 0 \\ 0 & \widetilde{A} \end{matrix} \right),\ \widehat B^{\operatorname{ou}\vert \operatorname{lin}} := \left(\begin{matrix} B\\ \widetilde{B}\end{matrix}\right), \  \widehat B^{\operatorname{ou}\vert \operatorname{lin}}_{\text{in}}:= \left(\begin{matrix}B_{\text{in}}\\ \widetilde{B}_{\text{in}}\end{matrix}\right),  \ \widehat C^{\operatorname{ou}\vert \operatorname{lin}}:=\left(\begin{matrix}C & -\widetilde C\end{matrix}\right), \\
& \widehat N^{\operatorname{lin}}:=\left(\begin{matrix}N & 0 \\ 0& \widetilde N \end{matrix} \right),  \widehat K^{\operatorname{ou}}_{\text{in}}:= \left(\begin{matrix}K_{\text{in}}\\ \widetilde{K}_{\text{in}}\end{matrix}\right), \text{and state variable } \widehat Z_t^{\operatorname{ou}\vert \operatorname{lin}}:= \left(\begin{matrix} Z_t\\ \widetilde{Z}_t\end{matrix}\right),
\end{split}
\end{equation}
where operators/states without tilde belong to System $1$, as in \eqref{eq:standardform}, and with tilde to some System $2$. This second system could be any other system with the same structure such as the reduced system, e.g., resulting from applying BT. Certainly, the output of the error system is the error between the outputs of both systems.}
Then one can define Gramians $\widehat O=\widehat W^{*}\widehat W$ and $\widehat P=\widehat R \widehat R^{*}$ of this error system \eqref{eq:compsys} that satisfy Lyapunov equations \eqref{eq:Lyapunoveq} or \eqref{eq:Lyapunoveq2} for the error system, i.e.
\begin{equation}
\begin{split}
&\widehat A^{*}\widehat O^{\operatorname{ou}} + \widehat O^{\operatorname{ou}} \widehat A  + \widehat C^{*}\widehat C = 0,  \\
&\widehat A( \widehat P^{\operatorname{ou}}-\widehat S)+ ( \widehat P^{\operatorname{ou}}-\widehat S)\widehat A^{*}+ \widehat B\widehat B^{*}=0
\end{split}
\end{equation}
where $\widehat S:=\widehat B_{\operatorname{in}} \widehat B_{\operatorname{in}}^{*}+ \widehat K \mathbb E(M_1 M_1^{*})\widehat K^{*}$ and analogously for linear systems with multiplicative noise. We can then perform a singular value decomposition $\widehat W\widehat R = \widehat V \Lambda \widehat U^{*}$ with diagonal operator $\Lambda$ that contains all singular values of the error system \eqref{eq:compsys} on its diagonal \cite[Theorem $5.1$]{ReSe}. It is then easy to check that 
\[\left\lVert \Delta(H) \right\rVert_{\operatorname{TC}} = \sum_{\lambda \in \Lambda} \lambda = \sum_{\mu \in \sqrt{\sigma(\widehat{O}\widehat{P})}} \mu.\]
This property follows as any decomposition $\widehat W \widehat R$ is equivalent to the Hankel operator $\widehat H$ associated with system \eqref{eq:compsys}:

More precisely, there exist unitary mappings \cite[Prop. $6.1$]{ReSe} $U: \overline{ \operatorname{ran}}(\widehat W\widehat R) \rightarrow \overline{ \operatorname{ran}}(\widehat H)$ and $V: \operatorname{ker}^{\perp}(\widehat W\widehat R) \rightarrow\operatorname{ker}^{\perp}(\widehat H)$ 
such that 
\[ \Delta(H) \vert_{\operatorname{ker}^{\perp}(\widehat H)} = \widehat H \vert_{\operatorname{ker}^{\perp}(\widehat H)}=U\left(\widehat W \widehat  R\right) \vert_{\operatorname{ker}^{\perp}(\widehat  W\widehat  R)} V^{*} \vert_{\operatorname{ker}^{\perp}(\widehat H)}. \]
{Notice that when BT is used, $\left\lVert \Delta(H) \right\rVert_{\operatorname{TC}}$ is expected to be small, since a reduced system is constructed that is supposed to keep the large Hankel singular values of the original system such that $ \sum_{\lambda \in \Lambda} \lambda$ has small summands in most of the cases.}

We summarize the preceding discussion of the Hankel operator error bounds:
\begin{itemize}
    \item The trace class norm of the Hankel operator difference is computable by solving in addition the Lyapunov equations for the error system consisting of the original and the reduced system \eqref{eq:compsys}.
    \item The error bound does not require the user to compute the Hankel operator directly. 
    \item As a word of caution: The Hankel operators do not have any obvious energy interpretation. In particular, the difference of Hankel operators in trace norm is not the same as the sum of truncated Hankel singular values in the MOR process.
\end{itemize}

\section{Ornstein-Uhlenbeck processes}
\label{sec:stoch}
Let $X$ be a Hilbert space, $A$ be the generator of a $C_0$-semigroup $(T_t)_{t \ge 0}$ on $X$, as well as $K:\mathbb R^d \rightarrow X$ and $B:\mathbb{R}^m \rightarrow X$ both linear and continuous maps.
For the OU processes \eqref{eq:OU}, we define the mild solution $(Z^{\operatorname{ou}}_t)_{t \ge 0}$ with initial state $\xi \in L^2(\Omega,\mathcal F_0,X)$ with output given by the variation of constant formula 
 \begin{equation}
 \label{eq:typea}
Y_t =  CZ^{\operatorname{ou}}_t  = CT_t\xi + \int_0^t CT_{t-s}K \ \mathrm dM_s + \int_0^t CT_{t-s} Bu_s \ \mathrm ds.
 \end{equation}

In particular, if $X$ is finite-dimensional or more general, if $(T_t)$ is uniformly continuous, then the semigroup is just given by $T_t:=e^{tA}.$

For OU processes we make the following stability assumption:
\begin{ass}[OU processes]
\label{ass:OU}
We assume that $A$ is the generator of an exponentially stable semigroup $(T_t)_{t \ge 0}$ such that for some $\omega>0$ and $\nu\ge 1: \ \Vert T_t \Vert \le \nu e^{-\omega t}.$ Moreover, we assume that $(M_t)_{t \ge 0}$ is a square-integrable mean zero  L\'evy process taking values in $\mathbb R^d$.
\end{ass}
{In the theory of balanced truncation, it is common to introduce two types of Gramians, an observability Gramian and a reachability Gramian. Here, we introduce for Ornstein-Uhlenbeck processes two possible types of such reachability Gramians, from which one of them is also taking into account non-zero initial conditions. Their slightly different definitions are mainly motivated by our two methods of obtaining error bounds that we introduce in this article.}
\begin{defi}[OU Gramians]
For the controlled OU process, we define the observability Gramian for $x,y \in X$ by
\begin{equation}
\begin{split}
\label{eq:OUGramians}
\langle x, \mathscr O^{\operatorname{ou}} y \rangle_X &:= \int_0^{\infty} \left\langle CT_s x,CT_s y \right\rangle_{\mathcal H} \ \mathrm ds
\end{split}
\end{equation}
and two types of reachability Gramians for $x,y \in X$ by
\begin{equation}
    \begin{split}
    \label{eq:OUGramians2}
\langle x,\mathscr P^{\operatorname{ou}}y \rangle_X &:= \int_0^{\infty} \langle x,T_t( K\mathbb E\left(M_1M_1^{*} \right)K^{*}+BB^{*}) T_t^{*}y \rangle_X \ \mathrm dt , \text{ and } \\
\langle x,\mathcal P^{\operatorname{ou}}y \rangle_X &:= \int_0^{\infty} \langle x,T_tBB^{*} T_t^{*}y \rangle_X \ \mathrm dt + \left\langle x,(B_{\operatorname{in}}B_{\operatorname{in}}^{*}+ K \mathbb E\left(M_1M_1^{*} \right)K^{*}) y \right\rangle_X.
\end{split}
\end{equation}
If $X$ is finite-dimensional, then the definition of the Gramians reduces in case of the observability Gramian to 
\begin{equation}
    \begin{split}
\mathscr O^{\operatorname{ou}} &= \int_0^{\infty}  T_s^{*}C^{*}CT_s  \ \mathrm ds
\end{split}
\end{equation}
and for the reachability Gramians to
\begin{equation}
    \begin{split}
\mathscr P^{\operatorname{ou}} &= \int_0^{\infty}  T_t ( K\mathbb E\left(M_1 M_1^{*} \right)K^{*}+BB^{*}) T_t^{*}  \ \mathrm dt, \text{ and } \\
\mathcal P^{\operatorname{ou}}&= \int_0^{\infty}T_tBB^{*} T_t^{*} \ \mathrm dt +B_{\operatorname{in}}B_{\operatorname{in}}^{*}+ K\mathbb E\left(M_1 M_1^{*} \right)K^{*}.
\end{split}
\end{equation}
\end{defi}
The weak formulation for infinite-dimensional spaces $X$ is needed in general, as $t\mapsto T_t$ is not necessarily measurable but $t \mapsto T_tx$ for any fixed $x \in X$ is.

\begin{defi}[OU Hankel operator]
\label{defi:OUHankel}
The \emph{OU Hankel operator} is the operator $H^{\operatorname{ou}}:=W^{\operatorname{ou}}R^{\operatorname{ou}} \in \mathcal L( L^2((0,\infty), \mathbb R^m) \oplus\mathbb R^d \oplus \mathbb R^k, L^2((0,\infty), \mathcal H))$.

Here, we assume that the controls take values in $\mathbb R^m$, the space of admissible initial states is $k$-dimensional, and the noise process takes values in $\mathbb R^d$. 

The \emph{observability map}  $W^{\operatorname{ou}} \in \mathcal L(X, L^2((0,\infty), \mathcal H))$ is defined as
\[ W^{\operatorname{ou}}_tx:=CT_tx \text{ such that }\mathscr O^{\operatorname{ou}} = W^{\operatorname{ou}*}W^{\operatorname{ou}}, \]
where $W^{\operatorname{ou}}$ is a Hilbert-Schmidt operator if $\mathcal H$ is finite-dimensional. 

The \emph{reachability map} $R^{\operatorname{ou}} \in \operatorname{HS}(   L^2((0,\infty), \mathbb R^m) \oplus\mathbb R^d \oplus \mathbb R^k,X)$ is defined as
\begin{equation}
\begin{split}
&R^{\operatorname{ou}}(f,v,u):= \int_0^{\infty} T_sBf_s \ \mathrm ds+ K\sqrt{\mathbb E(M_1M_1^{*})} v+ B_{\operatorname{in}}u \\
&\text{ such that } \mathcal P^{\operatorname{ou}}= R^{\operatorname{ou}}R^{\operatorname{ou}*}.
\end{split}
\end{equation}
\end{defi}
The Gramians \eqref{eq:OUGramians} and \eqref{eq:OUGramians2} satisfy the following Lyapunov equations:
\begin{prop}[Lyapunov equations]
\label{eq:LyapOU}
The observability Gramian satisfies for all $x_2,y_2 \in D(A)$
\begin{equation*}
\begin{split}
& \langle Ax_2, \mathscr O^{\operatorname{ou}}y_2 \rangle_X + \langle x_2, \mathscr O^{\operatorname{ou}}  Ay_2 \rangle_X+\langle x_2, C^{*}C y_2 \rangle_X  =0
\end{split}
\end{equation*}
and the reachability Gramians satisfies, for all $x_1,y_1 \in D(A^{*})$, with $S:=B_{\operatorname{in}}B_{\operatorname{in}}^{*}+ K\mathbb E(M_1M_1^*)K^{*}$,
\begin{equation*}
\begin{split}
&\left\langle  x_1,\mathscr P^{\operatorname{ou}} A^{*}y_1 \right\rangle_X +  \left\langle A^{*}x_1,\mathscr P^{\operatorname{ou}} y_1 \right \rangle_X  +\left\langle x_1,(BB^{*}+ K\mathbb E(M_1M_1^*)K^{*}) y_1 \right\rangle_X = 0 \text{ and } \\
&\left\langle  x_1,(\mathcal P^{\operatorname{ou}} -S)A^{*}y_1 \right\rangle_X +  \left\langle A^{*}x_1,(\mathcal P^{\operatorname{ou}} -S)y_1 \right \rangle_X+\left\langle x_1, BB^{*} y_1 \right\rangle_X = 0.
\end{split}
\end{equation*}
If $A$ is bounded, the equations reduce to 
\[A^{*}\mathscr O^{\operatorname{ou}} +  \mathscr O^{\operatorname{ou}}A+C^{*}C  =0\]
and
\begin{equation*}
\begin{split}
& \mathscr P^{\operatorname{ou}}A^{*}+ A\mathscr P^{\operatorname{ou}}+BB^{*}+ K\mathbb E(M_1M_1^*)K^{*}= 0\text{ and } \\
&(\mathcal P^{\operatorname{ou}} -S)A^{*}+ A(\mathcal P^{\operatorname{ou}} -S)+BB^{*}= 0.
\end{split}
\end{equation*}
\end{prop}
\begin{proof}
The Lyapunov equations follow immediately from the Lyapunov equations for linear deterministic systems \cite{ORW}: 
\medskip

This is immediate for the observability Gramian, since it coincides with the observability Gramian for linear systems. 

\medskip

For the reachability Gramian it suffices to observe that $\mathscr P^{\operatorname{ou}}$ and $\mathcal P^{\operatorname{ou}}-S$ are of the form of a linear reachability Gramian.
\end{proof}
\subsection{Error bounds}
We start by stating a direct bound for two OU processes as in \eqref{eq:OU} with $(C,A,K,B)$ and $(\widetilde C, \widetilde A, \widetilde K, \widetilde B)$, respectively, both having zero initial conditions. To this end, let $(T_t)$ and $(\widetilde T_t)$ be the semigroups generated by $A$ and $\widetilde A$.
To state the error bound, we introduce for $i \in \left\{1, 2 \right\}$ the auxiliary Gramians defined in terms of $B_1=B$ and $B_2 = K \sqrt{\mathbb E(M_1M_1^*)}$ by
\begin{equation}
\begin{split}
\label{Lyapunov_equations}
AP_i + P_i A^{*} &= -B_i B_i^{*}, \  P_i:= \int_0^\infty  C T_{s}B_iB_i^{*} T_{s}^{*} C^{*} \mathrm ds,\\
AP_{i, g}  + P_{i, g}  \widetilde{A}^{*} &= -B_i \widetilde{B}_i^{*}, \ P_{i, g} := \int_0^\infty  C T_{s} B_i\widetilde{B}_i^{*} \widetilde T_{s}^{*} \widetilde{C}^{*} \mathrm ds,\\
\widetilde{A}\widetilde{P}_i + \widetilde{P}_i \widetilde{A}^{*} &= -\widetilde{B}_i\widetilde{B}_i^{*}, \ \widetilde{P}_i := \int_0^\infty  \widetilde{C} \widetilde T_{s}\widetilde{B}_i\widetilde{B}_i^{*} \widetilde T_{s}^{*} \widetilde{C}^{*} \mathrm ds,
\end{split}
\end{equation}
and observe that the sums $\mathscr P^{\operatorname{ou}}=P_1+P_2 $ and $\widetilde{\mathscr P}^{\operatorname{ou}}=\widetilde P_1+\widetilde P_2$ coincide with the reachability Gramian for $B_{\text{in}}=0.$ Moreover, we write $\mathscr P^{\operatorname{ou}}_g=P_{1,g}+P_{2,g}.$ We then have the following error bound for the outputs of two Ornstein-Uhlenbeck processes starting from zero with possibly two different controls.

\begin{theo}[Error bound from zero]
\label{general_bound_same_u}
For control functions $u,\widetilde u \in L^2_{\operatorname{ad}}(\Omega_T,\mathbb R^n)$ and initial conditions $Y_0=\widetilde{Y}_0=0$, it follows that the difference between the outputs of two OU processes satisfies
\begin{equation}
\begin{split}
\label{eq:equation1}
\sup_{t \in [0, T]} \sqrt{\E\left[ \| Y_t - \widetilde{Y}_t\|^2 \right]} &\le \sqrt{2} (1 \vee \| u\|_{L^2(\Omega_T)}) \Big(\tr\Big(C \mathscr P^{\operatorname{ou}} C^{*} - 2C \mathscr P^{\operatorname{ou}}_g \widetilde{C}^{*} \\
&+\widetilde{C}\widetilde{\mathscr P}^{\operatorname{ou}}\widetilde{C}^{*}\Big) \Big)^{\frac{1}{2}} + \left(\operatorname{tr}(\widetilde C \widetilde{P}_1 \widetilde C^{*})\right)^{1/2} \Vert u - \widetilde u \Vert_{L^2(\Omega_T)}.
\end{split}
\end{equation}
\end{theo}
\begin{proof}
The explicit outputs of controlled OU processes are according to \eqref{eq:typea} given by 
\begin{align}
\label{OU_explicit_solutions}
\begin{split}
Y_t &= \underbrace{C \int_0^t T_{t-s}B u_s \ \mathrm ds}_{=:I_1(u)} + \underbrace{C \int_0^t T_{t-s}K \ \mathrm d M_s}_{=:I_2} \text{ and } \\
\widetilde{Y}_t &= \underbrace{\widetilde{C} \int_0^t \widetilde{T}_{t-s}  \widetilde{B} \widetilde{u}_s \  \mathrm ds}_{=:\widetilde{I}_1(\widetilde{u})} + \underbrace{\widetilde{C} \int_0^t \widetilde{T}_{t-s} \widetilde{K} \ \mathrm d M_s}_{=:\widetilde{I}_2}.
\end{split}
\end{align}

We insert the representations for $Y_t$ and $\widetilde{Y}_t$ from \eqref{OU_explicit_solutions} and obtain for \eqref{eq:equation1}
\begin{align}\label{eq:ineq_split_integrals}
\begin{split}
\left(\E \|Y_t-\widetilde{Y}_t\|^2\right)^{\frac{1}{2}}  &= \left(\E \| (I_1(u) - \widetilde{I}_1(u))+(\widetilde{I}_1(u) - \widetilde{I}_1(\widetilde u)) + (I_2 - \widetilde{I}_2) \|^2\right)^{\frac{1}{2}}  \\
&\le \left(\E \|I_1(u) - \widetilde{I}_1(u)\|^2\right)^{\frac{1}{2}} + \left(\E \|I_2 - \widetilde{I}_2\|^2\right)^{\frac{1}{2}}+\left(\E \|\widetilde I_1(u) - \widetilde{I}_1(\widetilde u)\|^2\right)^{\frac{1}{2}} 
\end{split}
\end{align}
From \cite[(31)]{FR2} we know that
\begin{equation}
\label{ineq_BM}
\mathbb E [\|I_2 - \widetilde{I}_2\|^2] \le \tr\left[C P_2 C^{*}- 2 C P_{2,g} \widetilde{C}^{*} + \widetilde{C}\widetilde{P}_2\widetilde{C}^{*}\right]
\end{equation}
for all $t \in [0, T]$. We can estimate the first term in \eqref{eq:ineq_split_integrals} using that
\begin{equation}
\label{ineq_control}
\begin{split}
\E [\|I_1(u) - \widetilde{I}_1(u)\|^2] &= \E\left[ \left\| \int_0^t \left( C T_{t-s}B u_s - \widetilde{C}\widetilde{T}_{t-s}\widetilde{B}\widetilde{u}_s  \right) \mathrm ds \right\|^2 \right] \\
&\le \E\left[ \left(\int_0^t \left\|  C T_{t-s}B- \widetilde{C} \widetilde{T}_{t-s}\widetilde B \right\|_{\operatorname{HS}} \|u_s\| \mathrm ds \right)^2 \right]\\
&\le \E\left[  \int_0^t \left\|  C T_{t-s}B  - \widetilde{C} \widetilde{T}_{t-s}\widetilde{B}  \right\|_{\operatorname{HS}}^2 \mathrm  ds \int_0^t\|u_s\|^2 \mathrm ds \right] \\
&=  \int_0^t \left\|  C T_{t-s}B  - \widetilde{C} \widetilde{T}_{t-s}\widetilde{B}  \right\|_{\operatorname{HS}}^2 \mathrm  ds \ \E \left[   \int_0^t\|u_s\|^2 \mathrm ds \right] \\
&\le \tr\left[CP_1 C^{*} - 2CP_{1, g}  \widetilde{C}^{*} + \widetilde{C}\widetilde{P}_1 \widetilde{C}^{*} \right] \| u\|^2_{L^2(\Omega_T)},
\end{split}
\end{equation}
where we used Cauchy-Schwarz and took the limit $t \to \infty$ in the first integral and $t \to T$ in the second one. Furthermore, we find for the remaining term in \eqref{eq:ineq_split_integrals} that \begin{align*}
\left(\E \|\widetilde I_1(u) - \widetilde{I}_1(\widetilde u)\|^2 \right)^{\frac{1}{2}} &\le \left(\E\left[\left(\int_0^t \left\|\widetilde{C} e^{\widetilde{A}(t-s)}\widetilde{B}_1 
\right\|_{\operatorname{HS}} \|u_s - \widetilde u_s\| \mathrm ds \right)^2 \right] \right)^{\frac{1}{2}}  \\
&\le \left(\int_0^t \left\|\widetilde{C} e^{\widetilde{A}(t-s)}\widetilde{B}_1  \right\|_{\operatorname{HS}}^2  \mathrm ds \right)^{\frac{1}{2}} \left(\E\int_0^t \|u_s - \widetilde u_s\|^2 \mathrm ds\right)^{\frac{1}{2}} \\
&\le\left(\tr(\widetilde{C}\widetilde{P}_1 \widetilde{C}^*) \right)^{\frac{1}{2}}   \| u - \widetilde u\|_{L^2(\Omega_T)}.
\end{align*}  
In order to get \eqref{eq:equation1}, we estimate
\begin{equation}\label{someset}
 \left(\E \|I_1(u) - \widetilde{I}_1(u)\|^2\right)^{\frac{1}{2}} + \left(\E \|I_2 - \widetilde{I}_2\|^2\right)^{\frac{1}{2}}
 \le  \sqrt{2}\sqrt{\E [\|I_1(u) - \widetilde{I}_1(u)\|^2]+ \E [\|I_2 - \widetilde{I}_2\|^2]}
\end{equation}
applying $a + b \le \sqrt{2}\sqrt{a^2 + b^2}$ for $a, b \in \mathbb{R}^+$. We insert \eqref{ineq_BM} and \eqref{ineq_control} into (\ref{someset}) and enlarge the resulting expression trough 
$1, \| u\|^2_{L^2(\Omega_T)}\leq (1 \vee \| u\|^2_{L^2(\Omega_T)})$. The bound \eqref{eq:equation1} now follows, by the linearity of the trace.
\end{proof}
A different control $\widetilde u$ in the reduced order model appears for example if model reduction is applied in the context of optimal control. Solving a control problem in the reduced system then
leads to a different control strategy compared to the full model. However, we see from the bound in Theorem \ref{general_bound_same_u} that the expression depending on the difference between $u$ and $\widetilde u$ is scaled by a term depending on $\widetilde P_1$, an operator that cannot 
be expected to be small. Hence, one can only guarantee a good approximation if $u$ and $\widetilde u$ are not too different. {Notice that the bound in Theorem \ref{general_bound_same_u} is a generalization of the result in \cite{FR2}, where $B=0$ was considered. Moreover, if the second system is a reduced model based on BT, then
$\tr\Big(C \mathscr P^{\operatorname{ou}} C^{*} - 2C \mathscr P^{\operatorname{ou}}_g \widetilde{C}^{*} +\widetilde{C}\widetilde{\mathscr P}^{\operatorname{ou}}\widetilde{C}^{*}\Big) $ in Theorem \ref{general_bound_same_u} can be expressed in terms of a weighted sum of truncated Hankel singular values of the system with zero initial data, which can be shown following the steps of \cite{FR2}. Therefore, the error of BT is low if we choose the reduced system dimension such that the truncated Hankel singular values are small.}

We now state an error bound in case the initial condition is not zero.
\begin{corr} [Error bound non-zero initial states]
\label{L2-error-nonzero-initial-split}
Let $u,\widetilde u \in L^2_{\operatorname{ad}}(\Omega_T,\mathbb R^n)$, $Y$ be the output of \eqref{eq:OU} with  $Z^{\operatorname{ou}}_0=\xi = B_{in} v$ and $\widetilde Y$ be the output of the reduced system 
with $\widetilde Z^{\operatorname{ou}}_0=0$. We define \begin{align}\label{definetildeY0}
\widetilde Y_t^{(0)} =\widetilde C^{(0)} \widetilde T^{(0)}_t \widetilde B_{in} v + \widetilde Y_t,\end{align}
where $(\widetilde T^{(0)}_t)_{t\geq 0}$ is a $C_0$-semigroup generated by some operator $\widetilde A^{(0)}$ and $\widetilde B_{in}$, $\widetilde C^{(0)}$ are additional input and output operators, respectively. Then, we have
\begin{align*}
\| Y - \widetilde{Y}^{(0)}\|_{L^2(\Omega_T)}  \le & \sqrt{2 T} (1 \vee \| u\|_{L^2(\Omega_T)}) \Big(\tr\Big(C \mathscr P^{\operatorname{ou}} C^{*} - 2C \mathscr P^{\operatorname{ou}}_g \widetilde{C}^{*} 
+\widetilde{C}\widetilde{\mathscr P}^{\operatorname{ou}}\widetilde{C}^{*}\Big) \Big)^{\frac{1}{2}} \\
&+ \|v\|_{L^2(\Omega)}\Big(\tr\Big(C P_0 C^{*} - 2C P_{0, g} \widetilde{C}^{(0)*} 
+\widetilde{C}^{(0)}\widetilde{P}_0\widetilde{C}^{(0)*}\Big) \Big)^{\frac{1}{2}} \\
&+\sqrt{T}  \left(\operatorname{tr}(\widetilde C \widetilde{P}_1 \widetilde C^{*})\right)^{1/2} \Vert u - \widetilde u \Vert_{L^2(\Omega_T)},
\end{align*}
where $P_0, P_{0, g}$ and $\widetilde P_0$ satisfy \begin{align*}
AP_0 + P_0 A^{*} = -B_{in} B_{in}^{*}, \  AP_{0, g}  + P_{0, g}  \widetilde{A}^{(0)*} = -B_{in} \widetilde{B}_{in}^{*}, \  
\widetilde{A}^{(0)}\widetilde{P}_0 + \widetilde{P}_0 \widetilde{A}^{(0)*} = -\widetilde{B}_{in}\widetilde{B}_{in}^{*}.
\end{align*}
\end{corr}
\begin{proof}
We use the triangle inequality to obtain
\begin{align*}
 \| Y - \widetilde{Y}^{(0)}\|_{L^2(\Omega_T)} \leq &\left(\mathbb E \int_0^T \|(Y_t - C T_t B_{in} v) - \widetilde Y_t\|^2 \mathrm dt\right)^{\frac{1}{2}} \\&
 + \left(\mathbb E\int_0^T \|C T_t B_{in} v - \widetilde C^{(0)} \widetilde T^{(0)}_t \widetilde B_{in} v  \|^2 \mathrm dt\right)^{\frac{1}{2}}.
\end{align*}
Since the function $Y_t - C T_t B_{in} v$, $t\in [0, T]$, is the output to \eqref{eq:OU} with zero initial state, Theorem \ref{general_bound_same_u} yields \begin{align*}
&\left(\mathbb E \int_0^T \|(Y_t - C T_t B_{in} v) - \widetilde Y_t\|^2 \mathrm dt\right)^{\frac{1}{2}} \leq  \sqrt{T}\left(\operatorname{tr}(\widetilde C \widetilde{P}_1 \widetilde C^{*})\right)^{1/2} \Vert u - \widetilde u \Vert_{L^2(\Omega_T)}\\
&\quad \quad\quad+ \sqrt{2T} (1 \vee \| u\|_{L^2(\Omega_T)}) \Big(\tr\Big(C \mathscr P^{\operatorname{ou}} C^{*} - 2C \mathscr P^{\operatorname{ou}}_g \widetilde{C}^{*} 
\widetilde{C}\widetilde{\mathscr P}^{\operatorname{ou}}\widetilde{C}^{*}\Big) \Big)^{\frac{1}{2}}.
                                          \end{align*}
Moreover, as in previous estimates, we find \begin{align*}
\mathbb E \int_0^T \|C T_t B_{in} v - \widetilde C^{(0)} \widetilde T^{(0)}_t \widetilde B_{in} v  \|^2 \mathrm dt &\leq
\int_0^T \|C T_t B_{in} - \widetilde C^{(0)} \widetilde T^{(0)}_t \widetilde B_{in}\|^2_{HS} \mathrm dt \;\mathbb E\|v\|^2 \\
&\leq \mathbb E\|v\|^2 \tr\Big(C P_0 C^{*} - 2C P_{0, g} \widetilde{C}^{(0)*}+\widetilde{C}^{(0)}\widetilde{P}_0\widetilde{C}^{(0)*}\Big)
\end{align*}
concluding the proof.
 \end{proof}
\begin{rem}\label{rem_splitting}
The choice of $\widetilde Y^{(0)}$ in (\ref{definetildeY0}) is motivated by the fact that \eqref{eq:OU} can be decomposed into a homogeneous and inhomogeneous part. Its output can then be written as
$Y_t = C \mathcal H^{\operatorname{ou}}_{t} + C \mathcal I^{\operatorname{ou}}_t$, where
\begin{align}
\label{eq:homOU}
d \mathcal H^{\operatorname{ou}}_{t} &= A \mathcal H^{\operatorname{ou}}_{t} \ \mathrm dt,\quad\mathcal H_{0}^{\operatorname{ou}}=\xi=B_{in} v, \\
\label{eq:perOU}
 d \mathcal I^{\operatorname{ou}}_t &= A\mathcal I^{\operatorname{ou}}_t \ \mathrm dt + Bu_t \ \mathrm dt + K\ \mathrm dM_t ,\quad \mathcal I_{0}^{\operatorname{ou}}=0.
\end{align}
As in \cite{BGM}, BT based on the Gramian $P_0$ can be applied to \eqref{eq:homOU} in order to get a reduced system with matrices $(\widetilde A^{(0)}, \widetilde B_{in}, \widetilde C^{(0)})$. BT
is used a second time but now based on $\mathscr P^{\operatorname{ou}}$ to find a reduced system to \eqref{eq:perOU}. The reduced order matrices in this case are $(\widetilde A, \widetilde B, \widetilde C, \widetilde K)$. The sum of 
both reduced order outputs is then a suitable candidate for the choice of $\widetilde Y^{(0)}$. {In the context of BT, it was also shown in \cite{BGM} that the error term $\tr\Big(C P_0 C^{*} - 2C P_{0, g} \widetilde{C}^{(0)*}+\widetilde{C}^{(0)}\widetilde{P}_0\widetilde{C}^{(0)*}\Big)$ is a function of the truncated Hankel singular values based on $P_0$. Consequently, BT applied to \eqref{eq:homOU} and \eqref{eq:perOU} yields a small error if one truncates the respective small Hankel singular values only.}
\end{rem}
{
We now state another error bound that takes into account the initial states and bounds the norms appearing in the control functional \eqref{eq:energy2}. In contrast to the previous approach in Remark \ref{rem_splitting}, the second ansatz does not rely on a splitting of the system. It is an all in one reduction procedure which invokes the Hankel operator that relies on the reachability Gramian $\mathcal P^{\operatorname{ou}}.$ However, the error will be bounded by the truncated singular values of the error system \eqref{eq:compsys} instead of the truncated Hankel singular values of the large-scale system. First, we need the following lemma,}
{where we employ $\Delta$ introduced in Subsection \ref{sec:Not}.}
\begin{lemm} 
\label{lemmas}
Let $\mathcal H \simeq \mathbb R^n$ be a finite-dimensional space, then for two systems with the same L\'evy noise profile, satisfying Assumption \ref{ass:OU}, the difference of their Hankel operators $\Delta(H^{\operatorname{ou}})$ satisfies 
\begin{equation}
\begin{split}
\label{eq:estm2}
\frac{1}{\sqrt{T}} \left\lVert  \Delta \left(  \int_0^t CT_{t-s} K \ \mathrm dM_s \right) \right\rVert_{L^2(\Omega_T,\operatorname{HS}(\mathbb R^m, \mathbb R^n))} 
&\le \left\lVert \Delta \left(H^{\operatorname{ou}} \right) \right\rVert_{\operatorname{HS}}, \\
 \left\lVert \Delta \left(CT B_{\operatorname{in}}\right)\right\rVert_{L^2(0,\infty),\operatorname{HS}(\mathbb R^k, \mathbb R^n))} &\le  \left\lVert \Delta\left(H^{\operatorname{ou}}\right) \right\rVert_{\operatorname{HS}}, \text{ and }\\
\left\lVert \Delta \left(CT B\right) \right\rVert_{L^1 ((0,\infty),\operatorname{HS}(\mathbb R^m, \mathbb R^n))} &\le 2\left\lVert \Delta \left(H^{\operatorname{ou}} \right) \right\rVert_{\operatorname{TC}}.
\end{split}
\end{equation}
\end{lemm}
\begin{proof}
 To obtain the first bound in \eqref{eq:estm2}, consider the process $X_t:=\int_0^t CT_{t-s}K \ \mathrm dM_s$ such that by Ito's isometry 
 \begin{equation}
 \begin{split}
 \label{eq:equality}
 \frac{1}{T} \left\lVert \Delta(X) \right\rVert_{L^2(\Omega_T)}^2 
 &= \frac{1}{T} \int_0^T \mathbb E \left\lVert \Delta(X_t) \right\rVert^2 \ \mathrm dt \\
 &\overset{\text{Ito's iso.}}{=}  \frac{1}{T}   \int_0^T \int_0^t  \Vert \Delta(CT_{t-s}K)\sqrt{\mathbb E(M_1M_1^*)} \Vert_{\operatorname{HS}}^2 \ \mathrm ds \ \mathrm dt\\ 
  &\overset{t-s \mapsto s}{=}  \frac{1}{T}   \int_0^T \int_0^t  \Vert \Delta(CT_sK)\sqrt{\mathbb E(M_1M_1^*)} \Vert_{\operatorname{HS}}^2 \ \mathrm ds \ \mathrm dt\\ 
 &\overset{(1)}{\le} \frac{1}{T}   \int_0^T \int_0^T  \Vert \Delta(CT_sK\sqrt{\mathbb E(M_1M_1^*)}) \Vert_{\operatorname{HS}}^2 \ \mathrm ds \ \mathrm dt\\
 &\overset{(2)}{=}   \int_0^T  \left\lVert \Delta( CT_s K \sqrt{\mathbb E(M_1M_1^*)} \right\rVert_{\operatorname{HS}}^2 \ \mathrm ds
 \end{split}
 \end{equation}
 where in $(1)$ we extended the integration range from $0$ to $T$ and in $(2)$ we used that the integrand is independent of $t.$
 
 \medskip
 
We now derive a lower bound on the Hilbert-Schmidt norm of the Hankel operator. 
Recall that the Hilbert-Schmidt norm of an operator is defined in \eqref{hsnorm}.

\medskip

Thus, using any ONB $(e_i)_{i \in \mathbb N}$ of $L^2((0,\infty), \mathbb R^n)$ and $(f_j)_{j \in \{1,..,d\}}$ of $\mathbb R^d$, we have the lower bound on the Hilbert-Schmidt norm, since we do not take a complete basis of the input space of the Hankel operator, yields the first estimate in \eqref{eq:estm2}
 \begin{equation}
 \begin{split}
 \label{eq:lb}
\left\lVert \Delta\left(H^{\operatorname{ou}}\right) \right\rVert_{\operatorname{HS}}^2
&\overset{\eqref{hsnorm}}{\ge} \sum_{j=1}^{d}\sum_{i=1}^{\infty}  \left\vert \left\langle \Delta\left(H^{\operatorname{ou}}\right)(0,f_j,0),e_i \right\rangle_{L^2} \right\rvert^2 \\
&\overset{\text{Def.}\ref{defi:OUHankel}}{=}  \sum_{j=1}^{d}\sum_{i=1}^{\infty}  \left\vert \left\langle \Delta \left(   CT_{\bullet} K \sqrt{\mathbb E(M_1M_1^*)} \right)f_j,e_i \right\rangle_{L^2} \right\rvert^2 \\
&\overset{\eqref{hsnorm}}{=}  \int_0^{\infty} \left\lVert \Delta (CT_sK\sqrt{\mathbb E(M_1M_1^*)}) \right\rVert_{\operatorname{HS}}^2 \ \mathrm ds  \overset{\eqref{eq:equality}}{\ge}\frac{1}{T} \left\lVert \Delta(X) \right\rVert_{L^2(\Omega_T)}^2. 
 \end{split}
 \end{equation}
 The second bound in \eqref{eq:estm2} follows straight from the definition of the Hilbert-Schmidt norm by taking an orthonormal basis $(e_i)_{i \in \mathbb N}$ of $L^2((0,\infty),\mathcal H)$ and $(f_i)_{i \in \{1,..,k\}}$ an orthonormal system of $\mathbb R^k$. Then, it follows that 
 \begin{equation}
 \begin{split}
 \Vert \Delta(H) \Vert_{\operatorname{HS}}^2 &\ge \sum_{i=1}^{\infty} \sum_{j=1}^k \vert \langle e_i, \Delta(H)(0,0,f_j)\rangle_{L^2} \vert^2 \\
 &=\left\lVert \Delta \left(CT B_{\operatorname{in}}\right)\right\rVert^2_{L^2((0,\infty),\operatorname{HS}(\mathbb R^k, \mathbb R^n))}.
  \end{split}
 \end{equation}
 The last bound in \eqref{eq:estm2} follows from linear BT theory \cite[Theorem $2.1$]{CGP}.
 \end{proof}
From the preceding estimates we can now obtain the following error bound on the global dynamics.
\begin{theo}[OU Error bound]
\label{theo:stoch}
Consider two OU-processes with the same control function $u \in L^2(\Omega_T,\mathbb R^m)$, see \eqref{eq:L2space}, driven by the same L\'evy processes, but (possibly different) initial conditions $\xi := \sum_{i=1}^k \langle v , \widehat{e_i} \rangle_{\mathbb{R}^k} \phi_i$ and $\widetilde{\xi}:= \sum_{i=1}^k \langle  v , \widehat{e_i} \rangle_{\mathbb{R}^k} \widetilde{\phi_i}$. Here, $(\phi_i)$ is the $L^2(\Omega,\mathcal F_0,X)$-orthonormal system of $B_{\operatorname{in}}.$ The difference between the outputs of two such processes satisfies 
\begin{equation}
\begin{split}
&\frac{\left\lVert \Delta \left(CZ^{\operatorname{ou}}\right) \right\rVert_{L^2(\Omega_T)}}{\sqrt T} \ \le  \left\lVert \Delta(H^{\operatorname{ou}}) \right\rVert_{\operatorname{TC}} \left(1+ \frac{ \left\lVert \xi \right\rVert_{L^{2}(\Omega)} + 2  \left\lVert u \right\rVert_{L^2(\Omega_T)}  }{\sqrt T} \right). 
\end{split}
\end{equation}
\end{theo}
\begin{proof}
We have for $v \in \mathbb R^k$ by orthonormality of $(\phi_i)$ that $\left\lVert v \right\rVert= \left\lVert \xi \right\rVert_{L^2(\Omega)}$ and define $X_t:=\int_0^t CT_{t-s}K \ \mathrm dM_s.$
By Young's inequality, which implies that for $f(s):= \left\lVert \Delta \left(\indic_{[0,\infty)}C T_s B \right) \right\rVert$ and $g(s):=\indic_{[0,T)} \left\lVert u_s \right\rVert$ we have 
\[ \Vert f*g \Vert_{L^2(0,T)}\le \Vert f \Vert_{L^1(0,T)} \Vert g \Vert_{L^2(0,T)},\] and Lemma \ref{lemmas}, it follows that 
\begin{equation*}
\begin{split}
 \left\lVert \Delta \left(CZ^{\operatorname{ou}}\right) \right\rVert_{L^2(\Omega_T)} 
 &\overset{\eqref{eq:typea}}{\le}  \Bigg(\left\lVert \Delta(X) \right\rVert_{L^2(\Omega_T)} +  \left\lVert \Delta(CT B_{\operatorname{in}})(v)\right\rVert_{L^2(\Omega_T)} \\
  &\quad\qquad \qquad \qquad  + \left\lVert \left\lVert \Delta \left(\indic_{[0,\infty)}C T B \right) \right\rVert * \indic_{[0,T)} \left\lVert u \right\rVert \right\rVert_{L^2(\Omega_T)} \Bigg)  \\
 &\overset{\text{Young's ineq.}}{\le}  \Bigg(\left\lVert \Delta(X) \right\rVert_{L^2(\Omega_T)} + \left\lVert \Delta(C T B_{\operatorname{in}})\right\rVert_{(0,\infty)}  \left\lVert \xi \right\rVert_{L^{2}(\Omega)} \\
 & \qquad \qquad \qquad\qquad \qquad  +  \left\lVert \Delta \left(C T B\right) \right\rVert_{L^1(0,\infty)} \left\lVert u \right\rVert_{L^{2}(\Omega_T)} \Bigg)\\
 &\overset{\text{Lemma }\ref{lemmas}}{\le}  \left\lVert \Delta(H^{\operatorname{ou}}) \right\rVert_{\operatorname{TC}} \left(\sqrt{T}+\left\lVert \xi \right\rVert_{L^{2}(\Omega)} + 2\left\lVert u \right\rVert_{L^2(\Omega_T)} \right).
 \end{split}
 \end{equation*} 

 \end{proof}
 
{We can see that the bound in Theorem \ref{theo:stoch} depends on $\left\lVert \Delta(H^{\operatorname{ou}}) \right\rVert_{\operatorname{TC}}$, which is the sum of singular values of the error system. By construction of BT, the associated reduced system keeps the larger Hankel singular values of the original system such that $\left\lVert \Delta(H^{\operatorname{ou}}) \right\rVert_{\operatorname{TC}}$ and hence the error is expected to be small whenever the second system is a reduced model by BT with appropriate reduced order dimension.}

\section{Linear systems with multiplicative noise}
\label{sec:bilstoch}
{{In this section, a bound for the output error between two S(P)DEs of the form \eqref{eq:bil} is proved. It is based on the singular values of the associated error system and therefore requires the study of suitable Gramians. This bound provides an a posteriori criterion for the approximation error, e.g., in the context of model order reduction. The solution to the linear S(P)DE is given as the sum of the homogeneous process satisfying}}
 \begin{equation}
\begin{split}
\label{eq:hom}
 &d \mathcal H^{\operatorname{lin}}_{t} = A \mathcal H^{\operatorname{lin}}_{t} \ \mathrm dt + N \mathcal H^{\operatorname{lin}}_{t} \ \mathrm dM_t,\text{ such that } \\
 &\mathcal H_{0}^{\text{lin}}=\xi
\end{split}
\end{equation}
and the solution to the inhomogeneous problem starting from zero
 \begin{equation}
\begin{split}
\label{eq:per}
 &d \mathcal I^{\operatorname{lin}}_t = A\mathcal I^{\operatorname{lin}}_t \ \mathrm dt + N \mathcal I^{\operatorname{lin}}_{t} \ \mathrm dM_t + Bu_t \ \mathrm dt,\text{ such that } \\
 &\mathcal I_{0}^{\operatorname{lin}}=0.
 \end{split}
\end{equation}
The solution to the homogeneous equation \eqref{eq:hom}, started at time $s$ from state $\xi$, defines a flow $\mathcal H_t^{\operatorname{lin}}=:\Phi^{\operatorname{lin}}_{t,s} \xi$. If the initial time is $s=0$, we just write $\Phi^{\operatorname{lin}}_t:=\Phi^{\operatorname{lin}}(t,0).$ 
{We now introduce a stability criterion for linear systems with multiplicative noise which is necessary to ensure dissipative dynamics.} 
\begin{ass}[Linear systems with multiplicative noise]
\label{ass:lin}
We make the assumption that $\Phi^{\operatorname{lin}}$ is exponentially stable in mean square sense, i.e.\@ there are $\gamma,c>0$ such that for all $\xi \in L^2(\Omega,\mathcal F_s,X)$ and $t \ge s$
 \begin{equation}
 \label{eq:flow}
  \mathbb E \left(\Vert (\Phi^{\operatorname{lin}}_{t,s} \xi) \Vert^2\right) \le  \gamma  e^{-c(t-s)} \mathbb E \Vert \xi \Vert^2.
  \end{equation}
Moreover, we assume that $(M_t)_{t \ge 0}$ is a square-integrable scalar-valued mean zero L\'evy process.
\end{ass}
We use the following representation of the homogeneous solution with flow $\mathcal H^{\operatorname{lin}}_{t}=:\Phi^{\operatorname{lin}}_t \xi$ such that
\begin{equation}
\begin{split}
\label{eq:process}
&CZ^{\operatorname{lin}}_t:=C\mathcal H^{\operatorname{lin}}_{t} +C\mathcal I^{\operatorname{lin}}_{t}  =C\Phi^{\operatorname{lin}}_t \xi + \int_0^t C \Phi^{\operatorname{lin}}_{t,s}B u_s \ \mathrm ds.
\end{split}
\end{equation}
This expression coincides with the output of the mild solution as discussed in \cite[(5.4)ff.]{BH19}.
The observability and reachability Gramian for linear systems with multiplicative noise are for $x,y \in X$ defined as
\begin{equation}
\begin{split}
\label{eq:bilGram}
\langle x, \mathscr O^{\operatorname{lin}} y \rangle_X &=\mathbb E \int_0^{\infty} \left\langle C \Phi^{\operatorname{lin}}_s x,C \Phi^{\operatorname{lin}}_s y \right\rangle_{\mathcal H} \ \mathrm ds \\
\langle x,\mathscr P^{\operatorname{lin}}y \rangle_X &= \mathbb E \int_0^{\infty} \left\langle x, (\Phi^{\operatorname{lin}}_s B)( \Phi^{\operatorname{bil}}_sB)^{*} y \right\rangle_X \ \mathrm ds + \left\langle x,B_{\operatorname{in}}B_{\operatorname{in}}^{*} y \right\rangle_X.
\end{split}
\end{equation}
To decompose the Gramians as 
\begin{equation}
\label{eq:decomposition}
\mathscr O^{\operatorname{lin}}=W^{\operatorname{lin}*}W^{\operatorname{lin}} \text{ and }\mathscr P^{\operatorname{lin}} =R^{\operatorname{lin}}R^{\operatorname{lin}*},
\end{equation}
we introduce observability $W^{\operatorname{lin}} \in \mathcal L(X,L^2(\Omega_{\infty}, \mathcal H))$ and reachability maps $R^{\operatorname{lin}} \in \operatorname{HS}( L^2(\Omega_{\infty}, \mathbb R^m) \oplus \mathbb R^k,X)$ defined as
\begin{equation}
\begin{split}
(W^{\operatorname{lin}}x)_t:=C\Phi^{\operatorname{lin}}_tx \text{ and } R^{\operatorname{lin}}(f,u):=\mathbb E \int_0^{\infty} \Phi^{\operatorname{lin}}_sBf_s \ \mathrm ds  + B_{\operatorname{in}}u.
\end{split}
\end{equation}
A straightforward computation shows that the above operators indeed satisfy \eqref{eq:decomposition}.
{The main theoretical tool for our study is the Hankel operator which we shall introduce next.}
\begin{defi}[Hankel operator]
\label{defi:bilHank}
The Hankel operator for the linear system with multiplicative noise is the Hilbert-Schmidt operator defined as 
\[ H^{\operatorname{lin}} := W^{\operatorname{lin}}R^{\operatorname{lin}} \in \operatorname{HS}(L^2(\Omega_{\infty}, \mathbb R^m) \oplus \mathbb R^k, \mathcal H) \]
and is trace-class if $\mathcal H$ is finite-dimensional.
\end{defi}
The above Hilbert-Schmidt and trace-class properties follow from the same arguments as in \cite[Sec. $5.2$]{BH19}. Adding the operator $B_{\operatorname{in}}$ to $R^{\operatorname{lin}}$ does not affect these properties as $B_{\operatorname{in}}$ is a finite rank operator.

\medskip

{The Gramians \eqref{eq:bilGram} satisfy Lyapunov equations given in the following proposition. This fact is very useful for the practical computation of these Gramians since such equations can be solved even in very high-dimensional settings.}
\begin{prop}[Lyapunov equations]
\label{prop:Lyapbil}
The stochastic Gramians for the system with multiplicative noise satisfy the following Lyapunov equations for all $x_1,y_1 \in D(A^{*})$ and $x_2,y_2 \in D(A)$
\begin{equation*}
\begin{split}
&\langle x_1, BB^{*} \ y_1 \rangle_X +  \langle A^{*}x_1,(\mathscr P^{\operatorname{lin}} -B_{\operatorname{in}}B_{\operatorname{in}}^{*})y_1  \rangle_X +\langle  x_1,\mathscr (\mathscr P^{\operatorname{lin}} -B_{\operatorname{in}}B_{\operatorname{in}}^{*})A^{*}y_1 \rangle_X  \\
&\qquad + \langle   N^{*}x_1,(\mathscr P^{\operatorname{lin}} -B_{\operatorname{in}}B_{\operatorname{in}}^{*}) N^{*}y_1  \rangle_X\ \mathbb E\left(M_1^2\right)=0  \text{ and } \\
&\langle x_2, C^{*}C y_2 \rangle_X + \langle Ax_2, \mathscr O^{\operatorname{lin}}y_2 \rangle_X + \langle x_2, \mathscr O^{\operatorname{lin}}  Ay_2 \rangle_X + \langle Nx_2, \mathscr O^{\operatorname{lin}}  Ny_2 \rangle_X \ \mathbb{E}(M_1^2) =0.
\end{split}
\end{equation*}
\end{prop}
\begin{proof}
It suffices to observe that the observability Gramian and $\mathscr P^{\operatorname{lin}} -B_{\operatorname{in}}B_{\operatorname{in}}^{*} $ coincide with the observability and reachability Gramian in \cite{BH19}. The Lyapunov equations are then stated in \cite[Lemma $5.6$]{BH19}.
\end{proof}
{Our next Lemma provides some auxiliary results that are relevant for the final error estimate of the difference of the stochastic dynamics in terms of the Hankel operator.}
\begin{lemm} 
\label{lemmabil}
Let $\mathcal H$ be a finite-dimensional space, we consider two linear multiplicative systems with the same or two i.i.d. square-integrable mean zero L\'evy processes $(M_t)_{t \ge 0}$ each, then the difference of Hankel operators $\Delta(H^{\operatorname{lin}})$ satisfies
\begin{equation}
\begin{split}
\label{eq:estmbil}
 \left\lVert \Delta \left(C\Phi^{\operatorname{lin}} B_{\operatorname{in}}\right)\right\rVert_{L^2(\Omega_{\infty},\operatorname{HS}(\mathbb R^k, \mathbb R^n))} &\le  \left\lVert \Delta\left(H^{\operatorname{lin}}\right) \right\rVert_{\operatorname{HS}}\text{ and }\\
\left\lVert \Delta \left(C\Phi^{\operatorname{lin}} B\right) \right\rVert_{L^1_tL^2_{\omega}(\Omega_{\infty},\operatorname{HS}(\mathbb R^m, \mathbb R^n))} &\le 2\left\lVert \Delta \left(H^{\operatorname{lin}} \right) \right\rVert_{\operatorname{TC}}.
\end{split}
\end{equation}
\end{lemm}
\begin{proof}
The first bound in \eqref{eq:estmbil} follows straight from the definition of the Hilbert-Schmidt norm, i.e. let $(f_j)_{j \in \{1,..,k\}}$ be an orthonormal basis of $\mathbb R^k$ and $(e_i)_{i \in \mathbb{N}}$ an orthonormal basis of $L^2(\Omega_{(0,\infty)},\mathcal H).$ This implies that 
\[ \left\lVert \Delta\left(H^{\operatorname{lin}}\right) \right\rVert^2_{\operatorname{HS}} \ge \sum_{j=1}^k \sum_{i=1}^{\infty} \vert \langle e_i,\Delta(H^{\operatorname{lin}})(0,f_j) \rangle_{L^2} \vert^2 = \left\lVert \Delta \left(C\Phi^{\operatorname{lin}} B_{\operatorname{in}}\right)\right\rVert_{L^2(\Omega_{\infty},\operatorname{HS}(\mathbb R^k, \mathbb R^n))}^2.\]
The second bound has been derived in \cite[Theorem $3$, (5.11)]{BH19} under the assumption that the noise profiles are independent. In the case of the same noise profile, the same proof as for \cite[Theorem $3$]{BH19} applies. This is because the flow of the coupled system $\widehat{Z_t}=(Z_t,\widetilde{Z_t})$ is a Markov process, which is the key property used in \cite[(5.12)]{BH19}.

The Markov property of $\widehat{Z_t}$ follows, since $\widehat{Z_t}$ is a solution to the S(P)DE
\begin{equation}
    \begin{split}
        d\widehat{Z_t}^{\operatorname{lin}} = \widehat{A}^{\operatorname{lin}} \widehat{Z_t}^{\operatorname{lin}} \ \mathrm dt + \widehat{N}^{\operatorname{lin}}\widehat{Z_t}^{\operatorname{lin}} \ \mathrm dM_t +  \widehat{B}^{\operatorname{lin}} u_t \ \mathrm dt,
    \end{split}
\end{equation}
where we used the notation introduced in \eqref{eq:compsys}. The solution to this system satisfies the Markov property \cite[Sec.\@$9.6$]{P15}.

\end{proof}
{We are now ready to state our main error bound.}
\begin{theo}[Error bound]
\label{theo:stochbil}
Consider two linear systems with multiplicative noise.
For initial conditions $\xi = \sum_{i=1}^k \langle v, \widehat{e_i} \rangle_{\mathbb{R}^k} \xi_i$ with $L^2(\Omega,\mathcal F_0,X)$-orthonormal system $(\xi_i),$ and $\widetilde{\xi}:= \sum_{i=1}^k \langle  v, \widehat{e_i} \rangle_{\mathbb{R}^k} \widetilde{\xi_i},$ it follows that for two L\'evy processes $(M_t)_{t \ge 0}$, which we assume to be either the same or independent, each one of them driving the dynamics of a linear system with multiplicative noise, we have for control functions $u \in L^2_{\operatorname{ad}}(\Omega_{\infty},\mathbb R^m)$ that  
\begin{equation}
\begin{split}
\left\lVert \Delta \left(CZ^{\operatorname{lin}}\right) \right\rVert_{L^2_tL^1_{\omega}(\Omega_{\infty})} \ &\le  \left\lVert \Delta(H^{\operatorname{lin}}) \right\rVert_{\operatorname{TC}} \left(\left\lVert \xi \right\rVert_{L^{2}(\Omega)} + 2 \left\lVert u \right\rVert_{L^2(\Omega_{\infty})}\right) 
\end{split}
\end{equation}
and for control functions $u \in L^2_tL^{\infty}_{\omega}(\Omega_{\infty},\mathbb{R}^n)$ we have 
\begin{equation} 
\begin{split}
\left\lVert \Delta \left(CZ^{\operatorname{lin}}\right) \right\rVert_{L^2(\Omega_{\infty})} \ &\le  \left\lVert \Delta(H^{\operatorname{lin}}) \right\rVert_{\operatorname{TC}} \left(\left\lVert \xi \right\rVert_{L^{2}(\Omega)} + 2 \left\lVert u \right\rVert_{L^{2}_{t}L^{\infty}_{\omega}(\Omega_{\infty})}\right)
\end{split}
\end{equation}
\end{theo}

\begin{proof}
From \eqref{eq:process} we find that
\begin{equation}
\begin{split}
\label{eq:estone}
&\left\lVert \Delta(CZ^{\operatorname{lin}}) \right\rVert_{L^2_t L^1_{\omega}(\Omega_{\infty})} \le \left\lVert \Delta(C\mathcal H^{\operatorname{lin}}) \right\rVert_{L^2_t L^1_{\omega}(\Omega_{\infty})} +\left\lVert \Delta(C\mathcal I^{\operatorname{lin}}) \right\rVert_{L^2_t L^1_{\omega}(\Omega_{\infty})} \text{ and } \\ 
&\left\lVert \Delta(CZ^{\operatorname{lin}}) \right\rVert_{L^2(\Omega_{\infty})} \le \left\lVert \Delta(C\mathcal H^{\operatorname{lin}}) \right\rVert_{L^2(\Omega_{\infty})} +\left\lVert \Delta(C\mathcal I^{\operatorname{lin}}) \right\rVert_{L^2(\Omega_{\infty})}.
\end{split}
\end{equation}

For the first terms on the right-hand side of \eqref{eq:estone} we have using 
\begin{itemize}
\item the Cauchy-Schwarz inequality in (1), 
\item the explicit expression for the homogeneous solution in (2), and 
\item the first estimate of \eqref{eq:estmbil} in (3)
\end{itemize}
that
\begin{equation}
\begin{split}
\label{eq:prelimestm}
\left\lVert \Delta(C\mathcal H^{\operatorname{lin}}) \right\rVert_{L^2_t L^1_{\omega}(\Omega_{\infty})} &\overset{(1)}{\le} \left\lVert \Delta(C\mathcal H^{\operatorname{lin}}) \right\rVert_{L^2(\Omega_{\infty})}\\
&\overset{(2)}{\le}  \left\lVert \Delta \left(C\Phi^{\operatorname{lin}} B_{\operatorname{in}}\right)\right\rVert_{L^2(\Omega_{\infty},\operatorname{HS}(\mathbb R^k, \mathbb R^n))}\left\lVert \xi \right\rVert_{L^{2}(\Omega)}\\
&\overset{(3)}{\le} \left\lVert \Delta\left(H^{\operatorname{lin}}\right) \right\rVert_{\operatorname{HS}} \left\lVert \xi \right\rVert_{L^{2}(\Omega)}.
\end{split}
\end{equation}

To estimate the second terms on the right-hand side of \eqref{eq:estone} we require some additional estimates on the inhomogeneous flow \eqref{eq:per}
\begin{equation}
\begin{split}
\label{eq:d2eq}
&\Vert \Delta(C \mathcal I^{\text{lin}}) \Vert_{L^2_t L^1_{\omega}(\Omega_{\infty})}^2\le\int_0^{\infty} \left(\mathbb E \int_0^t \left\Vert \Delta(C\Phi_{t,s}B)  \right\rVert \left\lVert u_s \right\rVert \ \mathrm ds \right)^2 \ \mathrm dt \\
&\overset{(1)}{\le} \int_0^{\infty} \left( \int_0^t \sqrt{\mathbb E(\left\Vert \Delta(C\Phi_{t,s}B)  \right\rVert^2)} \sqrt{\mathbb E(\left\lVert u_s \right\rVert^2)} \ \mathrm ds \right)^2 \ \mathrm dt \\
&\overset{(2)}{\le} \int_0^{\infty} \left( \int_0^t \sqrt{\mathbb E(\left\Vert \Delta(C\Phi_{t-s}B)  \right\rVert^2_{\operatorname{HS}})} \sqrt{\mathbb E(\left\lVert u_s \right\rVert^2)} \ \mathrm ds \right)^2 \ \mathrm dt \\
&\overset{(3)}{=} \int_{\mathbb{R}} \left( \int_{\mathbb R} \indic_{[0,\infty)}(t-s) \sqrt{\mathbb E(\left\Vert \Delta(C\Phi_{t-s}B)  \right\rVert_{\operatorname{HS}}^2)} \indic_{[0,\infty)}(s)\sqrt{\mathbb E(\left\lVert u_s \right\rVert^2)} \ \mathrm ds \right)^2 \ \mathrm dt.
\end{split}
\end{equation}
In (1) we applied H\"older's inequality in the expectation value and in (2) we use the Markov property, cf. \cite[(5.15)]{BH19}. In (3) we just rewrote the expression using indicator functions to make the convolutional structure more apparent.
If we then introduce auxiliary functions $f(s):=\indic_{[0,\infty)}(s) \sqrt{\mathbb E(\left\Vert \Delta(C\Phi_{s}B)  \right\rVert_{\operatorname{HS}}^2)}$ and $g(s):=\indic_{[0,\infty)}(s)\sqrt{\mathbb E(\left\lVert u_s \right\rVert^2)},$ we can interpret the above estimate as a convolution estimate
\[\Vert \Delta(C \mathcal I^{\text{lin}}) \Vert_{L^2_t L^1_{\omega}(\Omega_{\infty})} \le \Vert f*g \Vert_{L^2}.\]
If we then apply Young's convolution inequality we find
\[\Vert f*g \Vert_{L^2} \le \Vert f \Vert_{L^1} \Vert g \Vert_{L^2}.\]
Using that $\Vert f \Vert_{L^1} = \left\lVert \Delta \left(C\Phi^{\operatorname{lin}} B\right) \right\rVert_{L^1_tL^2_{\omega}(\Omega_{\infty},\operatorname{HS}(\mathbb R^m, \mathbb R^n))}$ and $\Vert g \Vert_{L^2}=\Vert u \Vert_{L^2(\Omega_{\infty})}$ and combining this with the second inequality in \eqref{eq:estmbil} yields
\begin{equation}
\begin{split}
\label{eq:d3eq}
\Vert \Delta(C\mathcal I^{\text{lin}}) \Vert_{L^2_t L^1_{\omega}(\Omega_{\infty})} 
&\le  \left\lVert \Delta \left(C\Phi^{\operatorname{lin}} B\right) \right\rVert_{L^1_tL^2_{\omega}(\Omega_{\infty},\operatorname{HS}(\mathbb R^m, \mathbb R^n))} \Vert u \Vert_{L^2(\Omega_{(0,\infty)})}\\
&\le 2\left\lVert \Delta \left(H^{\operatorname{lin}} \right) \right\rVert_{\operatorname{TC}}\Vert u \Vert_{L^2(\Omega_{\infty})}.
\end{split}
\end{equation}
Analogously, we find using Minkowski's integral inequality in (1) and analogous arguments as presented in estimates \eqref{eq:d2eq} and \eqref{eq:d3eq} to obtain (2) and (3) respectively, and using the second estimate in \eqref{eq:estmbil} to get (4) that  
\begin{equation}
\begin{split}
\label{eq:four}
&\Vert \Delta(\mathcal C\mathcal I^{\text{lin}}) \Vert_{L^2(\Omega_{\infty})}^2=\int_0^{\infty}\mathbb E \left( \int_0^t \left\Vert \Delta(C\Phi_{t,s}B)  \right\rVert \left\lVert u_s \right\rVert \ \mathrm ds \right)^2 \ \mathrm dt \\
&\overset{(1)}{\le} \int_0^{\infty} \left( \int_0^t \sqrt{\mathbb E(\left\Vert \Delta(C\Phi_{t,s}B)  \right\rVert^2)} \left\lVert u_s \right\rVert_{L^{\infty}(\Omega)} \ \mathrm ds \right)^2 \ \mathrm dt \\
&\overset{(2)}{\le} \int_0^{\infty} \left( \int_{\mathbb R} \indic_{(0,\infty)}(t-s) \sqrt{\mathbb E(\left\Vert \Delta(C\Phi_{t-s}B)  \right\rVert^2_{\operatorname{HS}})}\indic_{(0,\infty)}(s) \left\lVert u_s \right\rVert_{L^{\infty}(\Omega)} \ \mathrm ds \right)^2 \ \mathrm dt\\
&\overset{(3)}{\le} \left\lVert \Delta \left(C\Phi^{\operatorname{lin}} B\right) \right\rVert_{L^1_tL^2_{\omega}(\Omega_{\infty},\operatorname{HS}(\mathbb R^m, \mathbb R^n))}^2 \Vert u \Vert_{L^2_tL^{\infty}_{\omega}(\Omega_{\infty})}^2 \\
&  \overset{(4)}{\le} 4 \left\lVert \Delta \left(H^{\operatorname{lin}}\right) \right\rVert_{\operatorname{HS}}^2 \Vert u \Vert_{L^2_tL^{\infty}_{\omega}(\Omega_{\infty})}^2.
\end{split}
\end{equation}
Inserting bounds \eqref{eq:prelimestm}, \eqref{eq:d3eq}, \eqref{eq:four} into \eqref{eq:estone} then yields the claim.
\end{proof}
{We observe that the bounds of Theorem \ref{theo:stochbil} depend on $\left\lVert \Delta(H^{\operatorname{lin}}) \right\rVert_{\operatorname{TC}}$, which indicates once more that a reduced order model by BT will lead to a small error also in the case of multiplicative noise.}
We can (formally) improve our previous convergence result using interpolation to $q \in (1,2).$ The convex case $q=2$ will be analyzed separately in Section \ref{sec:LQR} for Wiener noise.

\begin{corr}
\label{corr:important}
Consider two linear systems with multiplicative noise profile that we assume to be either i.i.d. or the same for both systems.
For initial conditions $\xi = \sum_{i=1}^k \langle v, \widehat{e_i} \rangle_{\mathbb{R}^k} \xi_i$ with $L^2(\Omega,\mathcal F_0,X)$ orthonormal system $(\xi_i),$ and $\widetilde{\xi}:= \sum_{i=1}^k \langle  v, \widehat{e_i} \rangle_{\mathbb{R}^k} \widetilde{\xi_i}.$
 Let $q \in (1,2)$ then the following estimate holds
 \[ \Vert \Delta(C Z^{\operatorname{lin}}) \Vert_{L^2_t L^q_{\omega}(\Omega_T)} \le \Vert \Delta(C Z^{\operatorname{lin}}) \Vert^{2q^{-1}-1}_{L^2_t L^1_{\omega}(\Omega_T)} \Vert \Delta(C Z^{\operatorname{lin}}) \Vert^{2(1-q^{-1})}_{L^2(\Omega_T)}\]
 Moreover, we have that for any $T \in [0,\infty]$ that for $u \in L^2_{\operatorname{ad}}(\Omega_T)$ and $\gamma, c$ as in \eqref{eq:flow}
\begin{equation}
\label{eq:bd}
\Vert CZ^{\operatorname{lin}} \Vert_{L^2(\Omega_{\infty})} \le \gamma \Vert C \Vert  \left( \frac{\Vert \xi \Vert_{L^2(\Omega)}}{\sqrt{2c}} + \frac{\Vert B \Vert }{c} \Vert u \Vert_{L^2(\Omega_{\infty})}\right). 
 \end{equation}
It follows that for two L\'evy processes $(M_t)_{t \ge 0}$, that we assume either to be independent or the same, that drive the dynamics of a linear system with multiplicative noise, we have for control functions $u \in L^2_{\operatorname{ad}}(\Omega_{T})$ that  
\begin{equation}
\begin{split}
 \Vert \Delta(C Z^{\operatorname{lin}}) \Vert_{L^2_t L^q_{\omega}(\Omega_T)} \le &\left(\left\lVert \Delta(H^{\operatorname{lin}}) \right\rVert_{\operatorname{TC}} \left(\left\lVert \xi \right\rVert_{L^{2}(\Omega)} + 2 \left\lVert u \right\rVert_{L^2(\Omega_{T})}\right) \right)^{2q^{-1}-1} \times \\
  &\times \left(\gamma \Vert C \Vert  \left( \frac{\Vert \xi \Vert_{L^2(\Omega)}}{\sqrt{2c}} + \frac{\Vert B \Vert }{c} \Vert u \Vert_{L^2(\Omega_{\infty})}\right)\right)^{2(1-q^{-1})}.
  \end{split}
\end{equation}
\end{corr}
\begin{proof}
The result follows from applying H\"older's inequality twice: After applying H\"older's inequality in the expectation with parameters $p=(2-q)^{-1}$ and $\widetilde p=(q-1)^{-1}$ for $q$ as in the statement, we obtain
\begin{equation}
\begin{split}
\mathbb E\left(\Vert \Delta(CZ_t^{\operatorname{lin}}) \Vert^q \right) 
&\le \mathbb E\left(\Vert \Delta(CZ_t^{\operatorname{lin}}) \Vert^{2-q} \Vert \Delta(CZ_t^{\operatorname{lin}}) \Vert^{2(q-1)}\right)\ \\
&\le \left(  \mathbb E\left(\Vert \Delta(CZ_t^{\operatorname{lin}}) \Vert\right)  \right)^{2-q} \left(\mathbb{E}\left(\Vert \Delta(CZ_t^{\operatorname{lin}}) \Vert^2 \right)  \right)^{q-1}.
\end{split}
\end{equation}
We thus conclude that after applying H\"older's inequality with $p=(2q^{-1}-1)^{-1}$ and $\widetilde{p} = (2-2q^{-1})^{-1}$ in time that
\begin{equation}
\begin{split}
\label{eq:hoelders}
&\Vert \Delta(CZ^{\operatorname{lin}}) \Vert^2_{L^2_tL^q_{\omega}(\Omega_T)}=\int_0^T \mathbb E\left(\Vert \Delta(CZ_t^{\operatorname{lin}}) \Vert^q \right)^{2/q} \ \mathrm dt\\
&\le \int_0^T \left(  \mathbb E\left(\Vert \Delta(CZ_t^{\operatorname{lin}}) \Vert\right)  \right)^{4q^{-1}-2} \left(\mathbb{E}\left(\Vert \Delta(CZ_t^{\operatorname{lin}}) \Vert^2 \right)  \right)^{2(1-q^{-1})} \ \mathrm dt\\
&\le \left(\int_0^T \left(  \mathbb E\left(\Vert \Delta(CZ_t^{\operatorname{lin}}) \Vert\right)  \right)^{2} \ \mathrm dt\right)^{2q^{-1}-1}  \left(\int_0^T\left(\mathbb{E}\left(\Vert \Delta(CZ_t^{\operatorname{lin}}) \Vert^2 \right)  \right) \ \mathrm dt\right)^{2(1-q^{-1})} \\
&= \Vert \Delta(CZ^{\operatorname{lin}}) \Vert_{L^2_tL^1_{\omega}(\Omega_T)}^{2(2q^{-1}-1)}\Vert \Delta(CZ^{\operatorname{lin}}) \Vert_{L^2(\Omega_T)}^{4(1-q^{-1})}.
\end{split}
\end{equation}
It therefore suffices to verify the $L^2(\Omega_{T})$-boundedness of the process $CZ_t$, which is the second term in the last line of \eqref{eq:hoelders}, since the first term has been estimated in Theorem \ref{theo:stochbil}.

We then have from \eqref{eq:process} 
\begin{equation}
\begin{split}
\label{eq:estone1}
&\left\lVert \Delta(CZ^{\operatorname{lin}}) \right\rVert_{L^2(\Omega_{T})} \le \left\lVert \Delta(C\mathcal H^{\operatorname{lin}}) \right\rVert_{L^2(\Omega_{T})} +\left\lVert \Delta(C\mathcal I^{\operatorname{lin}}) \right\rVert_{L^2(\Omega_{T})}.
\end{split}
\end{equation}
The first term on the right-hand side, we can easily estimate as in \eqref{eq:prelimestm}
\begin{equation}
\begin{split}
\left\lVert \Delta(C\mathcal H^{\operatorname{lin}})\right\rVert_{L^2(\Omega_{T})}\le \left\lVert \Delta\left(H^{\operatorname{lin}}\right) \right\rVert_{\operatorname{HS}} \left\lVert \xi \right\rVert_{L^{2}(\Omega)}.
\end{split}
\end{equation}
Thus, it suffices to bound for $u \in L^2(\Omega_{\infty})$ the second term on the right-hand side of \eqref{eq:estone1}. This can be done by looking at 
\[ \mathcal I^{\operatorname{lin}}_t   = \int_0^t T_{t-s}N \mathcal I^{\operatorname{lin}}_s \ \mathrm dM_s + \int_0^t T_{t-s}Bu_s \ ds. \]

Recall that the solution is given 
\[CZ^{\operatorname{lin}}_t:=C\mathcal H^{\operatorname{lin}}_{t} +C\mathcal I^{\operatorname{lin}}_{t}  =C\Phi^{\operatorname{lin}}_t \xi + \int_0^t C \Phi^{\operatorname{lin}}_{t,s}B u_s \ \mathrm ds.\]
Using that the flow is exponentially stable, we find uniformly for all $T>0$
\[\Vert C\Phi^{\operatorname{lin}} \xi \Vert_{L^2(\Omega_T)} \le \gamma \Vert C \Vert \Vert \xi \Vert_{L^2(\Omega)} \sqrt{\int_0^T e^{-2ct} \ \mathrm dt} \le \gamma \frac{\Vert C \Vert \Vert \xi \Vert_{L^2(\Omega)}}{\sqrt{2c}}\]
and similarly for $X_t:= \int_0^t C \Phi^{\operatorname{lin}}_{t,s}B u_s \ \mathrm ds$ using exponential stability of the flow and Minkowski's integral inequality in (1) and Young's convolution inequality in (2)
\begin{equation}
\begin{split}
\Vert X \Vert_{L^2(\Omega_T)} 
&\overset{(1)}{\le}\sqrt{ \int_0^T \left(\int_0^t \gamma \Vert C \Vert \Vert B \Vert e^{-c(t-s)} \sqrt{\Vert \mathbb E \Vert u_s \Vert^2} \ \mathrm ds \right)^2 \ \mathrm dt } \\
& \overset{(2)}{\le} \frac{\gamma \Vert C \Vert \Vert B \Vert }{c} \Vert u \Vert_{L^2(\Omega_T)}.
\end{split}
\end{equation}
Thus, we have altogether that 
\[ \Vert CZ^{\operatorname{lin}} \Vert_{L^2(\Omega_T)} \le \gamma \Vert C \Vert  \left( \frac{\Vert \xi \Vert_{L^2(\Omega)}}{\sqrt{2c}} + \frac{\Vert B \Vert }{c} \Vert u \Vert_{L^2(\Omega_T)}\right). \]

\medskip

The final inequality in the statement of the Corollary then follows from the above estimates together with Theorem \ref{theo:stochbil}.

\end{proof}

\section{Optimal control theory}
\label{sec:OCT}
{
As discussed in the introduction of this paper, optimal control of large-scale SDEs \eqref{eq:standardform} given cost functionals \eqref{eq:energy2} is generally very expensive or even infeasible. Therefore, MOR is used to approximate these high-dimensional equations in order to subsequently solve the optimal control problem in the surrogate model. We denote the output of the full and the reduced system by $Y(u)= CZ(u)$ and  $\widetilde Y(u)= \widetilde C\widetilde Z(u)$, respectively, and write the dependence on the control $u$ explicitly. Obtaining an optimal control $\widetilde u_*$ in the reduced system, it is known by Theorems \ref{theo:stoch} and \ref{theo:stochbil} that $Y(\widetilde u_*)\approx \widetilde Y(\widetilde u_*)$ if we apply BT in order to ensure $\lVert \Delta(H^{\operatorname{ou}\vert \operatorname{lin}}) \rVert_{\operatorname{TC}}$ to be small. However, we are more interested in the performance of the reduced optimal control in the original system. This means that we measure the distances between $Y(u_*)$ and $ \widetilde Y(\widetilde u_*)$ as well as $u^*$ and $\widetilde u_*$ in terms of the cost functionals. Here, $u_*$ represents the optimal control in the original model. Therefore, we establish the following proposition. We start by showing that the abstract optimal control problems for the two stochastic equations \eqref{eq:standardform} with control functionals \eqref{eq:energy2} are well-posed. Moreover, we state explicit bounds on the cost functionals for the optimal control error under MOR.}
\begin{prop}
\label{prop:OU}
The optimal control problem \emph{(OCP)} for stochastic systems \eqref{eq:standardform} with associated energy functionals $J,$\footnote{We just write $J$ to denote any of the functionals in \eqref{eq:energy2}} as in \eqref{eq:energy2}, is well-posed and there exists a minimizer $u \in L^2(\Omega_T)$ to the \emph{OCP}. 
Let us now consider two systems, with outputs $CZ$ and $\widetilde C \widetilde Z$ satisfying the conditions of Theorems \ref{theo:stoch} and \ref{theo:stochbil} respectively, and consider two minimizers, of the two energy functionals systems given by
\begin{equation}
\begin{split}
u_* &= \operatorname{arg \ min}_{u} J (CZ(u),u,T) \text{ and }\widetilde{u}_* = \operatorname{arg \ min}_{u} J(\widetilde C \widetilde Z(u),u,T).
\end{split}
\end{equation}
In the case of Ornstein-Uhlenbeck processes we have
\begin{equation}
    \begin{split}
&\left\vert \sqrt{J_{\operatorname{LQR}}^{\operatorname{ou}}(CZ^{\operatorname{ou}}(u_{*}),u_{*},T)}-  \sqrt{J_{\operatorname{LQR}}^{\operatorname{ou}}(\widetilde{C}\widetilde{Z}^{\operatorname{ou}}(\widetilde u_{*}),\widetilde u_{*}),T)} \right\vert  \le \left\lVert \Delta(H^{\operatorname{ou}}) \right\rVert_{\operatorname{TC}}\times \\ &\quad \left(1+T^{-1/2}\left(\Vert \xi \Vert_{L^2(\Omega)} + 2\operatorname{max} \left\{\Vert u_{*}\Vert_{L^2(\Omega_T)},\Vert \widetilde{u}_{*}\Vert_{L^2(\Omega_T)}\right\}\right)\right)
\end{split}
\end{equation}
and for linear systems with multiplicative noise and $r \in [1,2)$ there is $C_T>0$ such that
\begin{equation}
    \begin{split}
&\left\vert \sqrt{J_r^{\operatorname{lin}}(CZ^{\operatorname{lin}}(u_{*}),u_{*},T)}-  \sqrt{J_r^{\operatorname{lin}}(\widetilde{C}\widetilde{Z}^{\operatorname{lin}}(\widetilde u_{*}),\widetilde u_{*},T)} \right\vert \\
& \qquad \le \left(\left\lVert \Delta(H^{\operatorname{lin}}) \right\rVert_{\operatorname{TC}} \left(\left\lVert \xi \right\rVert_{L^{2}(\Omega)} + 2 \operatorname{max}\{\left\lVert u_{*} \right\rVert_{L^2(\Omega_{\infty})},\left\lVert \widetilde{u}_{*} \right\rVert_{L^2(\Omega_{\infty})} \}\right) \right)^{2r^{-1}-1} \\
&\qquad\qquad \left(\gamma \Vert C \Vert  \left( \frac{\Vert \xi \Vert_{L^2(\Omega)}}{\sqrt{2c}} + \frac{\Vert B \Vert }{c}\operatorname{max}\{\Vert u_* \Vert_{L^2_{\operatorname{ad}}(\Omega_{T})},\Vert \widetilde{u}_* \Vert_{L^2_{\operatorname{ad}}(\Omega_{T})} \} \right)\right)^{2(1-r^{-1})}.
\end{split}
\end{equation}
\end{prop}
{\begin{rem}
The result of Proposition \ref{prop:OU} is an indicator that the optimal control $\widetilde u_*$ obtained from a reduced system  is of good quality if we choose the surrogate model such that $\lVert \Delta(H^{\operatorname{ou}\vert \operatorname{lin}}) \rVert_{\operatorname{TC}}$ is small. This can be ensured if a reduced system by BT with appropriate reduced order dimension is chosen.
\end{rem}}
\begin{proof}[Proof of Proposition \ref{prop:OU}]
We restrict ourselves, for the proof of the existence of minimizers, to systems \eqref{eq:bil}, as controlled OU processes \eqref{eq:OU} can be studied in a similar way.

Since the control functional is bounded from below, we can find a minimizing sequence of $u_n \in L_{\operatorname{ad}}^2(\Omega_T)$ defining processes $Z_n^{\operatorname{lin}}$ such that 
\[ \lim_{n \rightarrow \infty} J_r^{\operatorname{lin}}(CZ_n^{\operatorname{lin}}( u_{n}),u_{n},T) = \inf_{u_{*} \in L_{\operatorname{ad}}^2(\Omega_T)} J_r^{\operatorname{lin}}(CZ^{\operatorname{lin}}( u_{*}), u_{*},T)\] so that the $u_n$ satisfy
\begin{equation}
Z_{n}^{\operatorname{lin}}(t) = T_t\xi+ \int_0^t T_{t-s} N Z_{n}^{\operatorname{lin}}(s) \ \mathrm dM_s + \int_0^t T_{t-s}Bu_{n}(s)  \ \mathrm ds. 
\end{equation}
Since the $L_{\operatorname{ad}}^2(\Omega_T)$ norm of the elements $(u_n)$ is bounded, it follows  from \eqref{eq:bd} that $(Z_n^{\operatorname{lin}})$ is uniformly bounded in $L^2_{\operatorname{ad}}(\Omega_T)$.

\medskip

Weak compactness implies the existence of weak limits in $L^2_{\operatorname{ad}}(\Omega_T)$ for subsequences, that we denote just as the original sequences, $Z_n^{\operatorname{lin}} \rightharpoonup Z^{\operatorname{lin}} \in L^2_{\operatorname{ad}}(\Omega_T)$ and $u_n \rightharpoonup u \in L^2_{\operatorname{ad}}(\Omega_T).$

\medskip

Recall that by Ito's isometry and $\Vert T_{t-s} \Vert \le \nu e^{-\omega(t-s)}$ in (1), and Young's inequality (2) with $f(s):=\indic_{[0,\infty)}(s) \nu^2 e^{-2\omega s}$ and $g(s):=\indic_{[0,T)}(s)\mathbb E \Vert \mathscr S_s \Vert^2$ we have
\[ \Vert f*g \Vert_{L^1(\mathbb{R})} \le \Vert f \Vert_{L^1(\mathbb{R})}\Vert g \Vert_{L^1(\mathbb{R})} = \nu^2/(2\omega) \Vert \mathscr S \Vert^2_{L^2(\Omega_{T})}, \]
there exists a linear continuous operator
\begin{equation}
\begin{split}
&I: L^2_{\operatorname{ad}}(\Omega_{T}) \rightarrow L^2_{\operatorname{ad}}(\Omega_T),\quad I(\mathscr S)_t= \int_0^t T_{t-s} N \mathscr S_s \ \mathrm dM(s) \\
&\left\Vert  I(\mathscr S) \right\rVert_{L^2(\Omega_T)} \overset{(1)}{\le}  \sqrt{\mathbb{E}(M_1^2)} \Vert N \Vert \sqrt{\Vert f*g \Vert_{L^2(\mathbb{R})}} \overset{(2)}{\le}  \frac{\nu \sqrt{\mathbb{E}(M_1^2)} \Vert N \Vert}{\sqrt{2\omega}} \Vert \mathscr S \Vert_{L^2(\Omega_{T})}.
 \end{split}
 \end{equation}
Similarly, there is a continuous linear operator
 \begin{equation}
\begin{split}
&D: L^2_{\operatorname{ad}}(\Omega_T) \rightarrow L^2_{\operatorname{ad}}(\Omega_T), \quad D(u)_t= \int_0^t T_{t-s}B u_s \ \mathrm ds \\
& \left\lVert D(u)  \right\Vert_{L^2(\Omega_T)} \le \frac{\Vert B \Vert \nu}{\omega} \Vert u \Vert_{L^2(\Omega_T)}.
 \end{split}
 \end{equation}
 Thus, by weak convergence $Z_n^{\operatorname{lin}} \rightharpoonup Z^{\operatorname{lin}}$ in $L^2_{\operatorname{ad}}(\Omega_T)$, we can take any functional $f \in  L^2_{\operatorname{ad}}(\Omega_T)^*.$ Then $f\circ I \in L^2_{\operatorname{ad}}(\Omega_T)^*$ and thus the following weak limit exists
 \[ I(Z_n^{\operatorname{lin}}) \rightharpoonup I(Z^{\operatorname{lin}}) \text{ in }L^2_{\operatorname{ad}}(\Omega_T).\]
 
 Furthermore, we have the following weak limits in $L^2_{\operatorname{ad}}(\Omega_T)$
\begin{equation}
\begin{split}
\int_0^t T_{t-s} N Z_n^{\operatorname{lin}}(s) \ \mathrm dM_s  &\rightharpoonup \int_0^t T_{t-s} N Z^{\operatorname{lin}}_s \ \mathrm dM_s\text{ and } \\ 
\int_0^t T_{t-s}B u_{n}(s) \ \mathrm ds  &\rightharpoonup \int_0^t T_{t-s}Bu_s\ \mathrm ds
 \end{split}
 \end{equation}
 such that the process $Z^{\operatorname{lin}}$ satisfies with optimal control $u$
 \[Z^{\operatorname{lin}}_t  = T_t \xi +  \int_0^t T_{t-s} N Z^{\operatorname{lin}}_s \ \mathrm dM_s + \int_0^t T_{t-s}Bu_s \ \mathrm ds. \]
 Finally, to see that this solution actually minimizes the optimal control functional, we use that by weak convergence and lower semicontinuity of the norm
\begin{equation}
\begin{split}
\left\lVert C Z(u) \right\rVert^2_{L^2_tL^r_{\omega}(\Omega_T)}+ \langle u, R u \rangle_{L^2(\Omega_T)} &\le \inf_{u_{*} \in L_{\operatorname{ad}}^2(\Omega_T)} J_r^{\operatorname{lin}}(CZ^{\operatorname{lin}}( u_{*}), u_{*},T)\\
&\le \lim_{n \rightarrow \infty} J_{r}^{\text{lin}} (CZ(u_n),u_n,T)
  \end{split}
 \end{equation}
 which means that by the assumption on the sequence $u_n$, the control function $u$ is a minimizer.

We now write $Z(u)$ or $\widetilde{Z}(u)$ where $u$ is a control in order to emphasize which control is used. We then observe that from the inverse triangle inequality, we have for Ornstein-Uhlenbeck processes using \eqref{eq:energy2}
\begin{equation}
    \begin{split}
   \sqrt{J_{\operatorname{LQR}}^{\text{ou}}(\widetilde C \widetilde Z^{\text{ou}}(u_{*}),u_{*},T)}-T^{-1/2}\Vert \Delta\left(CZ^{\text{ou}}(u_*)\right) \Vert_{L^2(\Omega_T)}&\le \sqrt{J_{\operatorname{LQR}}^{\text{ou}}(CZ^{\text{ou}}(u_{*}),u_{*},T)}  \\
   \sqrt{J_{\operatorname{LQR}}^{\text{ou}}( C Z^{\text{ou}}(\widetilde u_{*}),\widetilde u_{*},T)}-T^{-1/2}\Vert \Delta\left(CZ^{\text{ou}}(\widetilde u_*)\right) \Vert_{L^2(\Omega_T)}&\le \sqrt{J_{\operatorname{LQR}}^{\text{ou}}(\widetilde C\widetilde Z^{\text{ou}}(\widetilde u_{*}),\widetilde u_{*},T)}
    \end{split}
\end{equation}
and for systems with multiplicative noise
\begin{equation}
    \begin{split}
   \sqrt{J_{r}^{\text{lin}}(\widetilde C \widetilde Z^{\text{lin}}(u_{*}),u_{*},T)}-\Vert \Delta\left(CZ^{\text{lin}}(u_*)\right) \Vert_{L^2(\Omega_T)}&\le \sqrt{J_{r}^{\text{lin}}(CZ^{\text{lin}}(u_{*}),u_{*},T)}  \\
   \sqrt{J_{r}^{\text{lin}}( C Z^{\text{lin}}(\widetilde u_{*}),\widetilde u_{*},T)}-\Vert \Delta\left(CZ^{\text{lin}}(\widetilde u_*)\right) \Vert_{L^2(\Omega_T)}&\le \sqrt{J_{r}^{\text{lin}}(\widetilde C\widetilde Z^{\text{lin}}(\widetilde u_{*}),\widetilde u_{*},T)}. 
    \end{split}
\end{equation}
Since $u_*$ and $\widetilde u_*$ are minimizers of the respective functional, we have 
\begin{equation}
    \begin{split}
  &J_{\operatorname{LQR}\vert r}^{\text{ou}\vert \text{lin} }(\widetilde C \widetilde Z^{\text{ou}\vert \text{lin} }(\widetilde u_{*}),\widetilde u_{*},T) \le   J_{\operatorname{LQR}\vert r}^{\text{ou}\vert \text{lin} }( \widetilde  C \widetilde  Z^{\text{ou}\vert \text{lin} }( u_{*}),  u_{*},T)\\
  &J_{\operatorname{LQR}\vert r}^{\text{ou}\vert \text{lin} }( C Z^{\text{ou}\vert \text{lin} }(u_{*}),u_{*},T) \le   J_{\operatorname{LQR}\vert r}^{\text{ou}\vert \text{lin} }( C Z^{\text{ou}\vert \text{lin} }(\widetilde  u_{*}), \widetilde u_{*},T).
  \end{split}
\end{equation}

Both estimates imply immediately that
\begin{equation}
    \begin{split}
&\left\vert \sqrt{J_{\operatorname{LQR}}^{\text{ou} }(\widetilde C \widetilde Z^{\text{ou}}(\widetilde u_{*}),\widetilde u_{*},T)}- \sqrt{J_{\operatorname{LQR}}^{\text{ou}}( C Z^{\text{ou}}(u_{*}),u_{*},T)} \right\vert \\
&\qquad \le T^{-1/2} \operatorname{max} \left\{\Vert \Delta\left(CZ^{\text{ou}}(\widetilde{u}_{*})\right) \Vert_{L^2(\Omega_T)},\Vert \Delta\left(CZ^{\text{ou}}(u_{*})\right) \Vert_{L^2(\Omega_T)}\right\} \text{ and } \\
&\left\vert \sqrt{J_{r}^{ \text{lin} }(\widetilde C \widetilde Z^{ \text{lin} }(\widetilde u_{*}),\widetilde u_{*},T)}- \sqrt{J_{r}^{ \text{lin} }( C Z^{ \text{lin} }(u_{*}),u_{*}T)}\right\vert \\
&\qquad  \le \operatorname{max} \left\{\Vert \Delta\left(CZ^{\text{lin}}(\widetilde{u}_{*})\right) \Vert_{L^2(\Omega_T)},\Vert \Delta\left(CZ^{\text{lin}}(u_{*})\right) \Vert_{L^2(\Omega_T)}\right\}.
\end{split}
\end{equation}
The bounds then follow from the conditions stated in Theorems \ref{theo:stoch} and \ref{theo:stochbil} respectively.

\end{proof}

\subsection{Infinite time Linear Quadratic Regulator}
\label{sec:LQR}
In the previous subsection we showed that the energy functionals \eqref{eq:energy2} with optimal control are well-approximated by the reduced order models, cf. Proposition  \ref{prop:OU}.

\medskip

In this subsection we go one step further and focus on the control itself and discuss techniques to approximate the optimal control using a reduced order model with a focus on infinite time horizons.

\subsubsection{Ornstein-Uhlenbeck processes}
\label{secsec:OUP}
Before discussing further the links between MOR and optimal control theory, we state in the next Proposition an approximation result on the optimal control $u$ to a high-dimensional Ornstein-Uhlenbeck process with Gaussian noise \eqref{eq:OU} and error control.
\begin{prop}
\label{prop:Ric}
Let $X$ be finite-dimensional and let $(Z^{\operatorname{ou}}_t)_{t \ge 0}$ be a controlled Ornstein-Uhlenbeck process \eqref{eq:OU} satisfying Assumption \ref{ass:OU} with standard Wiener noise $(W_t)_{t \ge 0}$ such that the pair $(A,C)$ is observable. 
The solution to the OCP with $T=\infty$ and $R>0$ in \eqref{eq:energy2}
is given by the fixed-point equation\footnote{as $Z^{\operatorname{ou}}_t$ itself depends on $u$}
\begin{equation}
\label{eq:outputu}
 u_{P}(t)=-R^{-1}B^{*}PZ^{\operatorname{ou}}_t,
 \end{equation}
where $P$ is the unique positive-definite solution to the Riccati equation
\begin{equation}
\label{eq:Riccati}
A^*P  + PA + C^{*}C-PBR^{-1}B^{*}P  = 0. 
\end{equation}
For a sequence $P_k\ge 0$ of unique solutions to standard Lyapunov equations
 \begin{equation}
\label{eq:sLyap}
A_k^{*}P_k+ P_k A_k +C^{*}C+ L_k^{*}RL_k=0 
\end{equation} 
where $L_k := R^{-1}B^{*}P_{k-1}$ for $k \ge 1$ and $A_{k} := A- BL_k$ for $k \ge 0$ with $L_0:=0$, matrices $P_k$ then converge quadratically and monotonically, in the sense of operators, to $P$. 
The control functions
\[ u_{P_k}(t)=-R^{-1}B^{*}P_kZ^{\operatorname{ou}}_t,\]
satisfy for $\Vert P- P_k \Vert$ sufficiently small, uniformly in the final time parameter $T$,
\[T^{-1/2} \Vert u_P-u_{P_k} \Vert_{L^2(\Omega_T)}= \mathcal O(\Vert P_k-P_{k-1} \Vert^2). \]
\end{prop}
\begin{proof}
Substituting \eqref{eq:outputu} into \eqref{eq:OU} yields an Ornstein-Uhlenbeck process
\begin{equation}
\begin{split}
\label{eq:condensed}
dZ^{\operatorname{ou}}_t = (A-BR^{-1}B^{*}P) Z^{\operatorname{ou}}_t \ \mathrm dt+ K \ \mathrm dW_t.
\end{split}
\end{equation}
The operator $A_P:=A-BR^{-1}B^{*}P$ is the generator of an exponentially stable semigroup $\Vert T_P(t) \Vert \le \nu e^{-\omega t}$ \cite[Theorem $1$]{Z} for some $\nu,\omega>0$.

Here, $P$ is the unique positive solution to the Riccati equation such that for all $x,y \in D(A)$
\begin{equation*} 
 \langle Ax,Py \rangle_X +\langle Px,Ay \rangle_X+ \langle (C^{*}C-PBR^{-1}B^{*}P)x,y \rangle_X = 0. 
 \end{equation*}
By Newton's method, one can approximate $P$ by a sequence $P_k \ge 0$, where $P_k$ solve Lyapunov equations \eqref{eq:sLyap}, for Hurwitz matrices $A_k$ \cite[Proof 1)]{Kleinman} with quadratic convergence rate \cite[(13)]{Kleinman} to the solution of the Riccati equation, namely 
\begin{equation}
\label{eq:newton}
\Vert P-P_k \Vert \le c \Vert P_k-P_{k-1} \Vert^2. 
\end{equation}
Standard results from semigroup theory imply that $A_k$ is also a generator with semigroup satisfying \cite[1.3, Chap.\@ 3]{EN}
\begin{equation}
\label{eq:Tk}
\Vert T_k(t) \Vert \le \nu e^{(-\omega+\nu\Vert B R^{-1} B^{*}\Vert \Vert P-P_k \Vert)t}.
\end{equation}
For an approximation $P_k$ of $P$ we find using \eqref{eq:typea}, \eqref{eq:outputu}, and \eqref{eq:condensed}
\begin{equation*}
\begin{split}
 &\Vert (u_P-u_{P_k})(t) \Vert_{L^2(\Omega)} \\
 &\le \Vert R^{-1}B^{*}\Vert \Vert P-P_k \Vert \Vert Z^{\operatorname{ou}}_t \Vert_{L^2(\Omega)}+ \Vert R^{-1}B^{*}P_k\Vert \Vert Z^{\operatorname{ou}}_t- Z^{\operatorname{ou}}_{k,t}\Vert_{L^2(\Omega)} \\
 &\le \Vert R^{-1} B^{*} \Vert \Bigg( \Vert P-P_k \Vert \Bigg(  \left\Vert \int_0^t T_{P}(t-s) K \ \mathrm dW_s \right\Vert_{L^2(\Omega)}+ \nu e^{-\omega t} \Vert \xi  \Vert_{L^2(\Omega)}  \Bigg) \\ 
&\quad +\Vert P_k \Vert \left( \left\Vert \int_0^t (T_P-T_{k})(t-s) K \ \mathrm dW_s \right\Vert_{L^2(\Omega)}+ \Vert  (T_P-T_{k})(t) \xi   \Vert_{L^2(\Omega)} \right)\Bigg).
\end{split}
\end{equation*}

We then use that by the product rule of differentiation
\begin{equation}
\begin{split}
\label{eq:product}(T_P-T_{k})(t) \xi &=-\int_0^t \frac{\mathrm d}{\mathrm ds}\left(T_P(t-s)T_k(s)\xi\right) \ \mathrm ds \\
&= -\int_0^t T_{P} (t-s) BR^{-1}B^{*}(P_k-P)T_{k}(s) \xi \ \mathrm ds
\end{split}
\end{equation}
such that due to \eqref{eq:Tk} and \eqref{eq:product}
\begin{equation*}
\begin{split}
\left\lVert (T_P-T_{k})(t) \xi \right\rVert_{L^2(\Omega)} 
&\le \nu^2 \Vert BR^{-1}B^{*} \Vert \Vert P-P_k \Vert  \times \\
& \qquad \int_0^t e^{-\omega(t-s)} e^{-(\omega -\nu  \Vert BR^{-1}B^{*} \Vert\Vert P-P_k \Vert)s} \ \mathrm ds \ \Vert \xi \Vert_{L^2(\Omega)} \\
&= \nu^2e^{-\omega t} \Vert BR^{-1}B^{*} \Vert \Vert P-P_k \Vert  \int_0^t  e^{\nu  \Vert BR^{-1}B^{*} \Vert\Vert P-P_k \Vert s} \ \mathrm ds \ \Vert \xi \Vert_{L^2(\Omega)} \\
&= \nu  e^{-\omega t}\left(e^{\nu \Vert BR^{-1}B^{*} \Vert\Vert P-P_k \Vert t}-1\right) \Vert \xi \Vert_{L^2(\Omega)}.
\end{split}
\end{equation*}
Rearranging and estimating further using Ito's isometry and the integral identity
\begin{equation}
\label{eq:auxint}
\int_0^{\infty} \left( e^{-a t} \left( e^{ct}-1 \right) \right)^2 \ \mathrm dt = \frac{c^2}{4a^3-6a^2c+2ac^2}, \text{ for } \Re(a)>\Re(c), \Re(a)>0,
\end{equation}
we obtain by setting $\alpha:=\nu\Vert BR^{-1}B^{*} \Vert$ 
\begin{equation}
\begin{split}
\Vert &(u_P-u_{P_k})(t) \Vert_{L^2(\Omega)} \le \nu \Vert R^{-1} \Vert\Vert B \Vert    \Bigg(\Vert P-P_k \Vert  \left( \Vert \xi  \Vert_{L^2(\Omega)} + \frac{\Vert K\Vert_{\operatorname{HS}}}{\sqrt{2\omega}}\right)\\
&+\frac{\alpha\Vert K\Vert_{\operatorname{HS}} \Vert P-P_k \Vert}{\sqrt{4\omega^3-6\omega^2\alpha\Vert P-P_k \Vert+2\omega \alpha^2 \Vert P-P_k \Vert^2}}+ e^{-\omega t}\left(e^{\alpha\Vert P-P_k \Vert t}-1\right) \Vert \xi \Vert\Bigg).  
\end{split}
\end{equation}
By taking the $L^2$ norm and regularizing the expression by dividing it by $\sqrt{T}$, we then finally obtain, using $T^{-1/2}\Vert 1 \Vert_{L^2(\Omega_T)}=1$
and \eqref{eq:auxint} in the last term, the following estimate
\begin{equation}
\begin{split}
T^{-1/2} \Vert u_P-u_{P_k} \Vert_{L^2(\Omega_T)} &\le  \nu \Vert R^{-1} \Vert\Vert B \Vert \Vert P-P_k \Vert \times \Bigg( \left( \Vert \xi  \Vert_{L^2(\Omega)}  \frac{\Vert K\Vert_{\operatorname{HS}}}{\sqrt{2\omega}}\right)  
 \\
 &\qquad +\frac{\alpha \left(\Vert K\Vert_{\operatorname{HS}}+T^{-1/2} \Vert \xi \Vert \right)}{\sqrt{4\omega^3-6\omega^2\alpha\Vert P-P_k \Vert+2\omega \alpha^2 \Vert P-P_k \Vert^2}}\Bigg).
\end{split}
\end{equation}
 \end{proof}
Thus, by approximating the solution to the Riccati equation using the scheme outlined in Proposition \ref{prop:Ric}, the optimal feedback law \eqref{eq:outputu} is approximated by the output of a new (uncontrolled) linear system 
\begin{equation}
\begin{split}
\label{eq:auxiliaryoutput}
dZ^{\operatorname{ou}}_t &= \bar{A}Z^{\operatorname{ou}}_t \ \mathrm dt+K \ \mathrm dW_t, \\
Z^{\operatorname{ou}}_0 &=\xi, \text{ and }\\
u_{P_k}(t) &= \bar{C} Z^{\operatorname{ou}}_t
\end{split}
\end{equation}
with operators 
\[\bar{C} = -R^{-1}B^{*}P_k, \ \bar{A}= A-BR^{-1}B^{*}P_k, \text{ and } \operatorname{ran}(B_{\operatorname{in}}) \ni \xi.\]
If we now define a reduced model to \eqref{eq:auxiliaryoutput}, e.g., by balancing the system (which is \eqref{eq:OU} with $(A, B)$ replaced by $(\bar A, 0)$ and output operator $\bar C$), we can use Theorem \ref{theo:stoch} to control the error between the outputs. This allows us to approximate the optimal control of the full high-dimensional system by the output of a reduced system of \eqref{eq:auxiliaryoutput}.

\medskip

The method outlined in this section allows us to approximate the (unique) optimal control of the full system using an auxiliary reduced order model. This is a stronger result than the approximation of energy functionals in Proposition \ref{prop:OU}. In general, the approximation of the optimal control may not be possible, since the optimal control may not be unique and may not be given as the output of a linear system, again.

\subsubsection{Linear systems with multiplicative noise}

We now turn to the infinite time OCP for finite-dimensional linear systems with multiplicative standard Wiener noise $(W_t)$ \eqref{eq:bil} and optimal control functionals \eqref{eq:energy2} with optimal control
\begin{equation}
\label{eq:OCPbil}
u_* = \operatorname{arg min}_{ u \in L^2(\Omega_T)}J_{\operatorname{LQR}}^{\text{lin}} (CZ^{\operatorname{lin}},u,\infty).
\end{equation}
Let $P$ then be the solution to the augmented Riccati equation \cite[(5)]{RZ}
\[ A^{*}P+PA+N^{*}PN-PBR^{-1}B^TP+C^{*}C=0. \]
The optimal control to \eqref{eq:OCPbil} is then given by the fixed-point equation ($Z^{\operatorname{lin}}$ also depends on $u_*$)
\begin{equation}
\label{eq:optcontrol}
 u_{*}(t) = -R^{-1} B^{*}P Z^{\operatorname{lin}}_t.
 \end{equation}
Thus, by replacing $u_*$ in the above expression by \eqref{eq:optcontrol}, we find that $u_*$ is the output of 
\begin{equation}
\begin{split}
\label{eq:522}
dZ^{\operatorname{lin}}_t &= \bar{A}Z^{\operatorname{lin}}_t \ \mathrm dt+NZ^{\operatorname{lin}}_t \ \mathrm dW_t \\
Z^{\operatorname{lin}}_0 &= \xi, \text{ and }\\
u_{*}(t) &= \bar{C}Z^{\operatorname{lin}}_t 
\end{split}
\end{equation}
with operators
\[ \bar{C} = -R^{-1}B^{*}P, \ \bar{A}= A-BR^{-1}B^{*}P, \text{ and } \operatorname{ran}({B}_{\operatorname{in}}) \ni \xi.\]

Reducing \eqref{eq:522} leads to an approximation for time optimal control that is based on solving a low-dimensional system.

\section{Numerical Examples}\label{sec:num}

\subsection{Controlled Ornstein-Uhlenbeck}

For an illustration of the above bounds we consider an Ornstein-Uhlenbeck process with control $u_t = \sin(t) \textbf{1} \in \mathbb R^d$ governed  by
\begin{equation}
\begin{split}
d Z_t &= AZ_t \ \mathrm dt + B_1 u_t \ \mathrm dt + B_2 \ \mathrm dW_t,\\
Y_t &= CZ_t, \ Z_0 = z_0,
\end{split}
\end{equation}
with $Z_t, W_t \in \mathbb{R}^d, A, B_1, B_2 \in \mathbb{R}^{d\times d}, C \in \mathbb{R}^{m \times d}, d, m = 50$, where we choose the corresponding matrices such that the dynamics is most pronounced in the first $r=5$ dimensions, namely
\begin{equation}
    A, B_1, B_2, C = \operatorname{diag}(\underbrace{-1, \dots, -1}_{r\text{ times}}, \underbrace{-0.01, \dots, -0.01}_{d-r\text{ times}}) + (\alpha_{ij}), 
\end{equation}
with random noise $\alpha_{ij} \sim \mathcal{N}(0, 10^{-6})$ i.i.d. being different for each variable. We either choose $z_0^* = (0, \dots, 0)$ or $z_0^* = (\underbrace{1, \dots, 1}_{r\text{ times}}, \underbrace{0, \dots, 0}_{d-r\text{ times}})$ as an initial value, take $B_{\operatorname{in}}=z_0$ and compare the bounds obtained in Theorems \ref{general_bound_same_u} and \ref{theo:stoch} and Corollary \ref{L2-error-nonzero-initial-split} with a simulation of the full and the reduced dynamics using BT.

\begin{figure}[H]

\centering
  \includegraphics[width=1.2\linewidth]{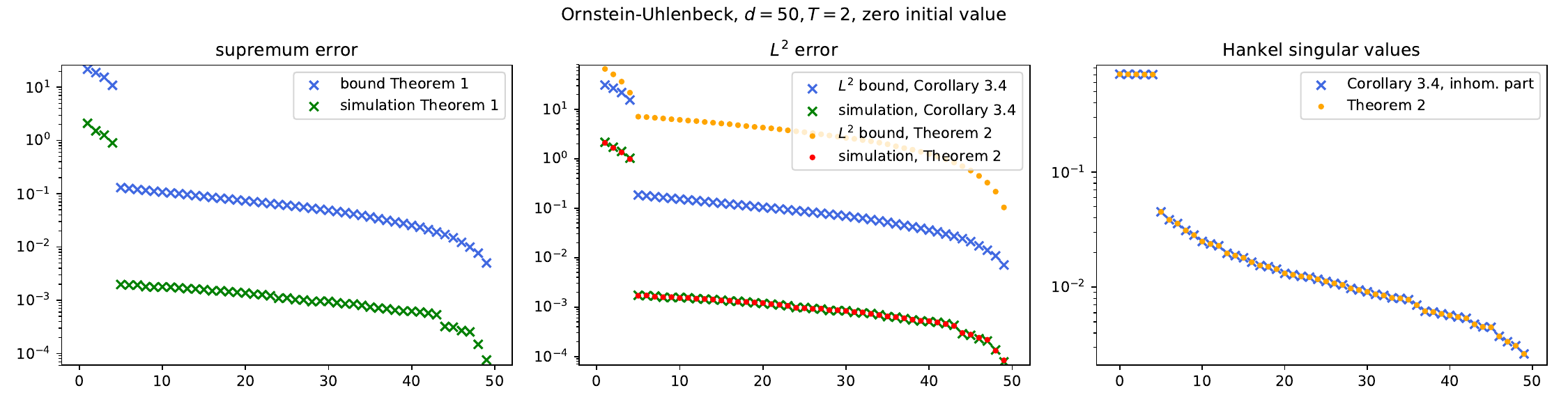}
  \includegraphics[width=0.8\linewidth]{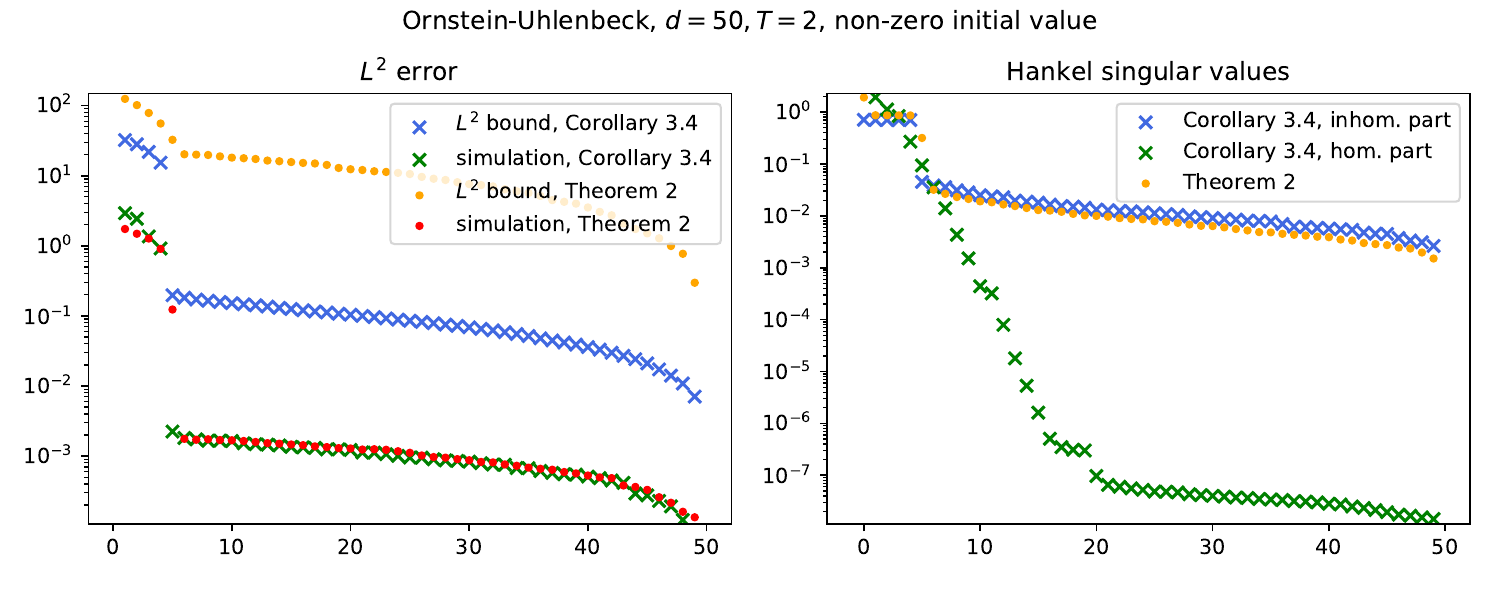}
\caption{Error bounds and simulations of BT of Ornstein-Uhlenbeck systems. The simulation is the numerically simulated error of the norm specified in the respective Theorem/Corollary.}
\label{OU_toy_example_BT} 
\end{figure}

In the top panel of Figure \ref{OU_toy_example_BT} we show the error bounds as well as the Hankel singular values and simulation results with varying dimension $r$ of the reduced model when starting in $z^*_0 = (0, \dots, 0)$. The simulation results are obtained with a simple Euler-Maruyama discretization with step-size $0.01$. We see that both bounds are rather conservative, the supremum bound on the left hand side seems to be a bit tighter than the $L^{2}$ bounds (also naturally due to the $\sqrt{T}$ scaling of the latter) and we in particular realize that the bound from Corollary \ref{L2-error-nonzero-initial-split} seems to be tighter than the one from Theorem \ref{theo:stoch}. The bottom panel shows the same approach, however, now choosing $z^*_0 = (1, \dots, 1, 0, \dots, 0)$. Here, we do not have a supremum bound anymore, but realize that the two $L^{2}$ bounds hold and that model reduction works well. For computing all the Gramians we use the formulas \eqref{eq:Lyapunoveq}. The code can be found at \url{github.com/lorenzrichter/balanced-truncation}.

 \subsection{Chain of oscillators}
 
 The one-dimensional chain of oscillators is a non-equilibrium statistical mechanics model that describes heat transport through a chain of $N$ particles coupled at each end to heat reservoirs at different temperatures with friction parameter $\gamma$ at the first and last particle. It was first introduced for the rigorous derivation of Fourier's law, or a rigorous proof of its breakdown: this is well described in \cite{BLR00}. We consider $N$ particles and denote by $q_i$ the location of each particle with respect to their equilibrium position and by $p_i$ its momentum. \\
The Hamilton function $H:\mathbb R^{2N} \rightarrow \RR$ of the system is given by
\begin{equation}
\begin{split}
H(\textbf{q},\textbf{p}) &= \frac{\langle \textbf{p},M^{-1}\textbf{p}\rangle}{2}+ V_{\eta,\zeta}(\textbf{q}), \text{ where } \\
V_{{\bf \eta,\zeta}}(\textbf{q}) &= \sum_{i=1}^N \eta_i q_i^2 + \sum_{i=1}^{N-1} \xi_i (q_i-q_{i+1})^2
\end{split}
\end{equation}
with mass matrix $M:=m \operatorname{id}_{\mathbb C^{N \times N}}$ and coupling strengths $\eta_i, \xi_i > 0.$
The above form of the potential describes particles that are fixed by a quadratic \textit{pinning} potential $U_{\operatorname{pin},i}(q) = \eta_i q^2 $ and interact with their nearest neighbors through a quadratic \textit{interaction} potential $U_{\operatorname{int},i}(q_i-q_j) =  \xi_i (q_i-q_j)^2$ for $j=i+1$ and $ i \in \{1,...,N\}$. \\
The $1^{\text{st}}$ and $N^{\text{th}}$ particle are each coupled to a heat bath at inverse temperatures $\beta_1$ and $\beta_N,$ respectively. We also assume these two particles $I=\left\{1,N \right\}$ to be subject to friction. The dynamics of the system is described by the Langevin dynamics
\begin{align*}
\mathrm d \mathbf q_t &= M^{-1}\mathbf  p_t  \ \mathrm d t \\
\mathrm d \mathbf  p_t &= (-S \mathbf  q_t - \Gamma \mathbf p_t + \sigma u_t) \ \mathrm d t + \sigma \ \mathrm d W_t 
\end{align*} 
where $u_t \in \mathbb R^N$ is an external control and $(W_t)$ an $\RR^N$-valued standard Wiener process.
Expressing the system using phase-space coordinates $\mathbf Z_t:=(\mathbf q_t^*, \mathbf p_t^*)^*$ we see that the entire system is described by the Ornstein-Uhlenbeck process 
\begin{equation}
\label{eq:system of SDEs}
\mathrm d  \mathbf Z_t = (A \mathbf Z_t +  B u_t) \ \mathrm dt +  K\ \mathrm d W_t
\end{equation}
with
\begin{equation}
\begin{split}
\label{eq:ourchoice}
&A = \begin{pmatrix} 0 & M^{-1} \\ -S & -\Gamma \end{pmatrix}, 
\quad K = B= \begin{pmatrix} 0 & 0 \\ 0&  \sigma \end{pmatrix}, \text{ with fluctuation-dissipation relation}\\ &\sigma = \operatorname{diag}\left(\frac{\sqrt{2m\gamma}}{\sqrt{\beta_1}},0,\dots,0,\frac{\sqrt{2m\gamma}}{\sqrt{\beta_N}}\right), \quad \Gamma = \operatorname{diag}(\gamma, 0,\dots, 0,\gamma).
\end{split}
\end{equation}
Here, we changed the notation so that $u_t \in \mathbb R^{2N}$ is an external control and $(W_t)$ an $\RR^{2N}$-valued standard Wiener process.

The operator $S$ is the Jacobi (tridiagonal) matrix for $f=(f_1,...,f_N) \in \RR^N$, defined as
\[(S f)_n = -\xi_n f_{n+1}- \xi_{n-1}f_{n-1} +(\eta_n+(2-\delta_{n \in I})\xi_n)f_n\]
where $f_0=f_{N+1}:=0.$
The matrix $A$ is Hurwitz if all parameters of the model are strictly positive.

The invariant distribution to the uncontrolled process \eqref{eq:system of SDEs} is given by \cite{LLR}
\begin{equation}
\label{eq:mu}
\mu_{\Sigma_{\beta}}(\textbf{q},\textbf{p}):= (2\pi)^{-N/2} \operatorname{det}(\Sigma_{\beta}^{-1/2}) \operatorname{exp} \left( -\tfrac{1}{2} \langle (\textbf{q},\textbf{p}),\Sigma_{\beta}^{-1} (\textbf{q},\textbf{p}) \rangle \right),
\end{equation}
where the covariance matrix $\Sigma_{\beta}$ is the solution to the Lyapunov equation \cite[(2.8)]{LLR}
\begin{equation}
\label{eq:Lyapu}
A\Sigma_{\beta}+\Sigma_{\beta}A^{*} +KK^* =0. 
\end{equation}

\subsection{Friction and spectral gap}

If in the chain of oscillators one chooses the friction according to \eqref{eq:ourchoice}, then the spectral gap of $A$ closes necessarily as $N \rightarrow \infty.$ This is apparent by studying $$\sum_{\lambda \in \sigma(A)} \lambda= \operatorname{tr}(A)  = \operatorname{tr}(-\Gamma) = -2\gamma.$$
Since we have $2N$ (counting multiplicity) eigenvalues with negative real parts, we conclude that the one with largest real part decays to zero at least with rate $\vert \Re(\lambda_S) \vert = \mathcal O(N^{-1})$.

\medskip

The situation changes once we apply a constant non-zero friction $\gamma:=\gamma_1=\gamma_2>0$ such that $\Gamma:=\operatorname{diag}\left(\gamma,\dots, \gamma \right)$ to all the particles. In this case, we find for the determinant using the block-determinant formula
\[\operatorname{det}\begin{pmatrix}  Q_{11}&Q_{12} \\ Q_{21}&Q_{22} \end{pmatrix}=\operatorname{det}(Q_{22}Q_{11}-Q_{21}Q_{12}) \text{ if } Q_{11}Q_{12}=Q_{12}Q_{11}\]
the decomposition
\[\operatorname{det}(A-\lambda I)=\operatorname{det}(\lambda^2 I +\lambda \Gamma+SM^{-1})=0.\]

This equation is equivalent to solving $\lambda^2+ \gamma \lambda +\mu=0$ where $\mu \in \sigma(SM^{-1}).$ By explicitly solving the quadratic equation, one can see that this equation has only solutions with strictly negative real part if $SM^{-1}$ has a uniform -- in the number of particles -- spectral gap. A comprehensive discussion of the spectral gap for this model can be found in \cite{M19,BM18}.

\medskip

For our numerical simulations we do not want the closing of the spectral gap to inflict the simulations. We therefore consider a mild constant friction parameter $\gamma_2$ and a larger friction parameter $\gamma_1$ at the terminal ends of the chain. To be precise, we choose a simulation time $T = 10$, $N = 75$ oscillators and $\gamma_1 = 10, \gamma_2 = 0.25, m =\xi_n = \eta_n = \beta_1 = \beta_N = 1$. Figure \ref{chain_of_oscillators_BT_plot} shows the BT bound from Theorem \ref{general_bound_same_u} on the left hand side and the $L^2$ bound from \ref{theo:stoch} in the middle subplot along with the simulated errors, again with varying reduced dimension $r$ on the $x$-axis. The plot on the right hand side shows the Hankel singular values. We can see that indeed one can reduce the dimension of the system significantly with only getting a small error and we note that the $L^2$ error bound seems to saturate for large $r$, which might be due to numerical issues.

\begin{figure}[H]

\centering
  \includegraphics[width=1.2\linewidth]{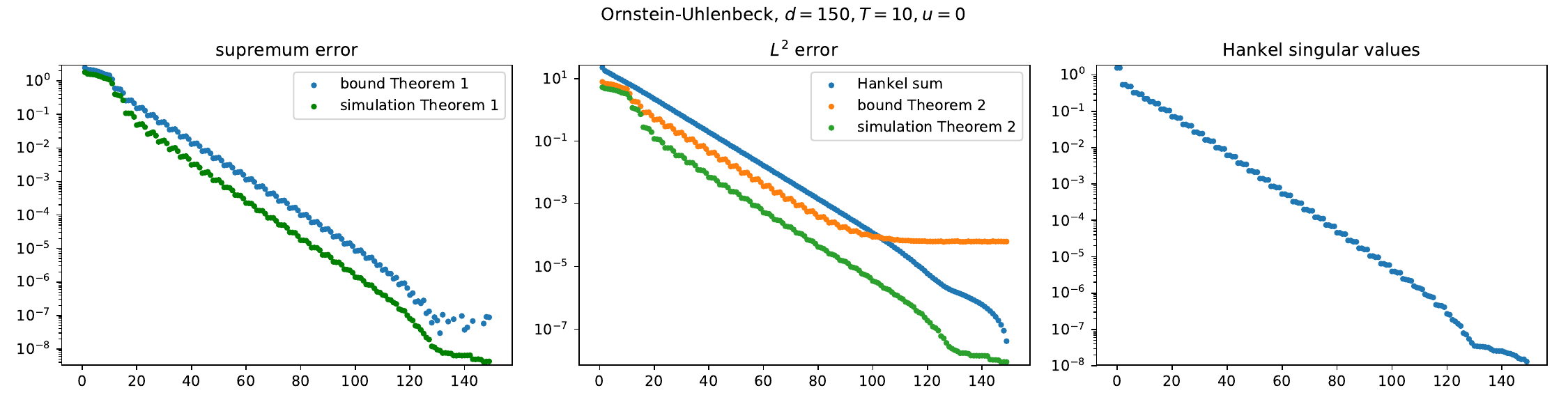}
  \caption{Error analysis of the chain of oscillators when applying BT.}
  \label{chain_of_oscillators_BT_plot}
\end{figure}

\subsection{Stochastic optimal control}
\label{sec:steer}

We now study the set of reachable distributions $\mathcal{N}(0, \Sigma)$ for a controlled OU process \eqref{eq:OU}. To be precise, we are looking for a feedback law of minimal energy 
\begin{equation}
\label{eq:LQRchain}
J_{\operatorname{LQR}}^{\text{ou}} (0,u,\infty):=   \lim_{T \rightarrow \infty} \frac{\left\lVert u  \right\rVert^{2}_{L^2(\Omega_T)}}{T} 
\end{equation}
to maintain an invariant state $\mu_{\Sigma}$ for some given $\Sigma>0$, namely
\begin{equation}
\label{eq:states}
\mu_{\Sigma}(\textbf{q},\textbf{p}):= (2\pi)^{-N} \operatorname{det}(\Sigma^{-1/2}) \operatorname{exp} \left( -\tfrac{1}{2} \langle (\textbf{q},\textbf{p}),\Sigma^{-1} (\textbf{q},\textbf{p}) \rangle \right).
\end{equation}
According to \cite[Theorem 4]{CGP2} this invariant state can be attained with a control $u^{*}_t = -K^{*} \Pi\mathbf Z_t$, where $\Pi$ is (any) symmetric matrix that satisfies
\begin{equation}
\label{Pavon_Pi_equation}
(A - B B^{*} \Pi) \Sigma + \Sigma (A - B B^{*} \Pi)^{*} + K K^{*} = 0.
\end{equation}

In our next Proposition we show that, from the invariant distribution for the chain of oscillators associated with some boundary temperatures $\beta=(\beta_1,\beta_N)$, we can reach the invariant state associated with any other boundary temperature $\beta'=(\beta_1',\beta'_N).$
\begin{prop}
There exists a control that steers the chain of oscillators \eqref{eq:system of SDEs}, with physical temperature $\beta=(\beta_1,\beta_N)$, to the invariant distribution $\mathcal N(0,\Sigma_{\beta'})$ with temperatures $\beta'=(\beta_1',\beta'_N).$ If $\beta_1=\beta_n$ and $\beta_1'=\beta'_N$ then the invariant state has covariance matrix
\begin{equation}
\label{eq:Sigmabeta}
 \Sigma_{\beta'} = \beta_1'^{-1} \begin{pmatrix} S^{-1} & 0 \\ 0 & M \end{pmatrix}
 \end{equation}
 and a solution $\Pi$ to \eqref{Pavon_Pi_equation} reads
\begin{equation}
\label{eq:Pi}
 \Pi= \operatorname{diag}\left(0,\frac{(\beta_1'-\beta_1)}{2}M^{-1}\right). 
 \end{equation}
\end{prop}
\begin{proof}
A sufficient condition \cite[Theorem $4$]{CGP2} to be able to reach a state $\mathcal N(0,\Sigma_{\beta'})$ is that $\operatorname{im}(B) \subset \operatorname{im}(K)$ and $\Sigma$ solves the Lyapunov equation
\[ \Sigma_{\beta'} A^*+A\Sigma_{\beta'}+KK^*+BX^{*}+XB^{*}=0 \]
for some $X.$
We thus define diagonal matrices $X_{\delta}$ for $\delta_1,\delta_N \in \RR$ by
\begin{equation}
\begin{split}
X_{\delta}&:=\operatorname{diag} ( \underbrace{0,...,0}_{n \text{ times}}, \delta_1,\underbrace{0,...,0}_{n-2 \text{ times}}, \delta_N).
\end{split}
\end{equation}
It is then obvious that for a suitable choice of $\delta$ and any other temperature $\beta'=(\beta_1',\beta_N')$ at the terminal ends of the chain we have due to \eqref{eq:Lyapu}
\[ A\Sigma_{\beta'}+\Sigma_{\beta'}A^{*} +K_{\beta'}K_{\beta'}^* =0  \]
such that by choosing $\delta$ such that
 \[A K_{\beta}K_{\beta}^{*}+ B_{\beta}X_{\delta}^{*}+X_{\delta}B_{\beta}^{*} = K_{\beta'}K_{\beta'}^{*} \]
where we used the subscript $\beta$ to emphasize the temperature profile used in the respective matrix.
 This implies that the uncontrolled chain of oscillators \eqref{eq:system of SDEs} with equilibrium state \eqref{eq:mu} and temperature $\beta$ can be steered into the equilibrium state \eqref{eq:mu} for any other temperature $\beta'.$
 
 \medskip
 
 The form of the covariance matrix \eqref{eq:Sigmabeta} can be directly verified by inserting it into \eqref{eq:Lyapu}.

To verify \eqref{eq:Pi}, we use the fluctuation-dissipation relation $\sigma \sigma^{*} = \frac{2}{\beta_1} M \Gamma$ and write the symmetric matrix $\Pi$ as a block matrix 
\[\Pi = \begin{pmatrix} \Pi_{11} & \Pi_{12} \\ \Pi_{21} & \Pi_{22} \end{pmatrix},\]we then get
\begin{equation*}
\begin{split}
&\left(\begin{pmatrix} 0 & M^{-1} \\ -S & -\Gamma \end{pmatrix} - \begin{pmatrix} 0 & 0 \\ 0 & \sigma \sigma^{*} \end{pmatrix} \begin{pmatrix} \Pi_{11} & \Pi_{12} \\ \Pi_{21} & \Pi_{22} \end{pmatrix} \right) \begin{pmatrix} S^{-1} & 0 \\ 0 & M \end{pmatrix} \\
&\ + \begin{pmatrix} S^{-1} & 0 \\ 0 & M \end{pmatrix} \left( \begin{pmatrix} 0 & -S \\ M^{-1} & -\Gamma \end{pmatrix} -  \begin{pmatrix} \Pi_{11}^{*} & \Pi_{21}^{*} \\ \Pi_{12}^{*} & \Pi_{22}^{*} \end{pmatrix} \begin{pmatrix} 0 & 0 \\ 0 & \sigma \sigma^{*} \end{pmatrix}  \right) = -\beta_1' \begin{pmatrix} 0 & 0 \\ 0 & \sigma \sigma^{*} \end{pmatrix}
 \end{split}
\end{equation*}
which reduces to
\begin{equation*}
\begin{split}
& \begin{pmatrix} 0 & M^{-1} \\ -S-\sigma \sigma^{*} \Pi_{21} & -\Gamma - \sigma \sigma^{*} \Pi_{22}  \end{pmatrix}    \begin{pmatrix} S^{-1} & 0 \\ 0 & M \end{pmatrix} \\
&+ \begin{pmatrix} S^{-1} & 0 \\ 0 & M \end{pmatrix}   \begin{pmatrix} 0 & -S-\Pi_{21}^{*}\sigma\sigma^{*} \\ M^{-1} & -\Gamma - \Pi_{22}^{*}\sigma \sigma^{*}   \end{pmatrix} \\
&= - \beta_1'\begin{pmatrix} 0 & 0 \\ 0 & \sigma \sigma^{*} \end{pmatrix}.
  \end{split}
\end{equation*}
From the block $(12)$ we get
\begin{equation*}
M^{-1} M - S^{-1}S - S^{-1} \Pi_{21}^{*} \sigma \sigma^{*} = 0 \text{ such that we can choose } \Pi_{21} = 0.
\end{equation*}
From the block $(22)$ we get
\begin{equation*}
 \beta_1'^{-1}\left(-\Gamma M - \sigma \sigma^{*} \Pi_{22} M - M \Gamma - M \Pi_{22}^{*} \sigma \sigma^{*} \right)= - \sigma \sigma^{*}.
\end{equation*}
By symmetry, $\Pi_{12}=0$. One can check that
\begin{equation*}
\Pi_{22}  = \frac{(\beta_1'-\beta_1)}{2}M^{-1}.
\end{equation*}
At last, we may then choose $\Pi_{11}=0$ since this matrix does not enter in the Lyapunov equation. 

\medskip

To see that our choice of $\Pi$ is admissible it remains to verify that $A-KK^*\Pi$ is Hurwitz. This however follows immediately since $A$ is Hurwitz and $-KK^*\Pi$ is diagonal with non-positive entries.

\end{proof}
\par\bigskip
\begin{rem}
If one wants to solve \eqref{Pavon_Pi_equation} for a general covariance matrix $\Sigma$, vectorization can be used to get 
\begin{align*}
\operatorname{vec}(A \Sigma + \Sigma A^{*} + KK^{*}) &= \operatorname{vec}(BB^{*} \Pi \Sigma + \Sigma \Pi^{*} BB^{*})\\
&= (\Sigma \otimes BB^{*}) \operatorname{vec}(\Pi) + (BB^{*} \otimes \Sigma) \operatorname{vec}(\Pi^{*}) \\
&= (\Sigma \otimes BB^{*} + BB^{*} \otimes \Sigma) \operatorname{vec}(\Pi),
\end{align*}
since we assume $\Pi$ to be symmetric. Note that $\Pi$ is admissible only if the rank condition
\begin{equation*}
    \operatorname{rank}\begin{pmatrix} A\Sigma + \Sigma A^{*} + K K^{*} & B \\ B^{*} & 0 \end{pmatrix} = \operatorname{rank}\begin{pmatrix} 0 & B \\ B^{*} & 0 \end{pmatrix}
\end{equation*}
holds and $A - BB^{*} \Pi$ is Hurwitz (see \cite{CGP2}).
\end{rem}

\subsection{Optimal control meets balanced truncation}

We now discuss how to use BT to steer subsystems into a designated steady state.

We again consider the high-dimensional Ornstein-Uhlenbeck process \eqref{eq:system of SDEs}, for which we have discussed in Subsection \ref{sec:steer} the convergence of
\[ \Sigma_t = \mathbb E(Z_t Z_t^{*}) \]
to a designated covariance matrix $\Sigma>0,$ under certain conditions.

Now, we want to study the case where we only want to find a control that maintains a certain covariance matrix $\mathbb{R}^{r \times r}\ni \Sigma_{rr}>0$ for an $r \ll d$-dimensional projection of our original system. In this case, the above method does not apply immediately. 

\medskip

To be precise, we are interested in reaching the sub-covariance matrix $\Sigma_{rr}$ as the limiting covariance matrix of
\[  \mathcal Q\Sigma_t  \mathcal Q^{*} =\mathcal Q\mathbb E(Z_t Z_t^{*})\mathcal Q = \mathbb E((\mathcal QZ_t) (\mathcal QZ_t)^{*}),\]
where $\mathcal Q$ is a suitable projection matrix.

\medskip

We can now first reduce the model to $r$ dimensions (recall that $r$ is the rank of $\mathcal Q$) using BT with observability matrix $C=\mathcal Q$ and then apply the method described in Subsection \eqref{sec:steer} to the reduced system
$(\widetilde C,\widetilde A,\widetilde K,\widetilde B)$ by using that 
\[ \mathcal Q\Sigma_t \mathcal Q^{*} \approx \mathbb E\left(\widetilde C \widetilde Z_t (\widetilde C \widetilde Z_{t})^{*}\right).\]
More precisely, it follows that 
\begin{equation}
\begin{split}
&\left\lVert \mathbb E\left( (CZ_t)(CZ_t)^* \right)- \mathbb E\left( (\widetilde C\widetilde Z_t)(\widetilde C\widetilde Z_t)^*\right) \right\rVert \\
&\le \left\lVert \mathbb E\left( (CZ_t-\widetilde C\widetilde Z_t)(CZ_t)^*\right) \right\rVert+\left\lVert \mathbb E\left( (\widetilde C\widetilde Z_t)((CZ_t)^*-(\widetilde C\widetilde Z_t)^*)\right) \right\rVert \\
&\le \Vert  CZ_t-\widetilde C\widetilde Z_t \Vert_{L^2(\Omega)} \left(\Vert  CZ_t\Vert_{L^2(\Omega)}+ \left\lVert \widetilde C\widetilde Z_t \right\Vert_{L^2(\Omega)} \right).
\end{split}
\end{equation}
Thus, the covariance matrix $\E\left(\widetilde Z_{t}\widetilde Z_{t}^{*}\right)$ that the reduced process $\widetilde Z_{t}$ is supposed to maintain is the normal distribution \eqref{eq:states} with (formal inverse) $\Sigma^{-1} = \widetilde C^* \Sigma_{rr}^{-1}\widetilde C$. If $\Sigma^{-1}$ has full rank, and thus $\Sigma^{-1}$ is the inverse of an actual matrix $\Sigma$, then this auxiliary distribution for the reduced system can be used to compute an optimal control, as described in Section \ref{sec:steer}, for the full system.

\medskip

We illustrate the above ideas in the following example.

\begin{ex}[Target distribution of outmost oscillators.]
Let us say we want to prescribe the covariance matrix of the subsystem containing only the leftmost and rightmost oscillators and accordingly choose $\mathcal Q \in \mathbb{R}^{4\times d}$, $d=2N,$ with $\mathcal Q_{11} = 1, \mathcal Q_{2,N} = 1, \mathcal Q_{3, N+1} = 1, \mathcal Q_{4, 2N} = 1$, to retain position and momentum variables, and choose all other $\mathcal Q_{ij} = 0$. We can then employ BT to obtain a reduced system associated with the original system
\begin{equation}
\begin{split}
dZ_t &= (A Z_t+Bu_t) \ \mathrm dt + K_{\beta} \ \mathrm dW_t\\
Y_t&=\mathcal QZ_t.
\end{split}
\end{equation}
The reduced system is of lower dimension $r$ with $r  \ll d$,
\begin{equation}
\begin{split}
d\widetilde{Z}_t &= (\widetilde A \widetilde Z_t+\widetilde Bu_t) \ \mathrm dt + \widetilde K_{\beta} \ \mathrm dW_t\\
\widetilde Y_t&=\widetilde{\mathcal Q}\widetilde Z_t.
\end{split}
\end{equation}
To run a numerical simulation we choose the sub-covariance to be 
\begin{align}
    \Sigma_{kk} = S_{kk} + S_{kk}^{*}, \quad S_{kk} =\operatorname{diag}(3, \dots, 3) + (|a_{ij}|), \quad a_{ij} \sim \mathcal{N}(0, 1)
\end{align}
and compute the optimal control as described above. We have realized that it is important to actually check the speed of convergence as \cite{CGP2} does not say anything about the time needed to be ``close'' to the stationary distribution. This can for instance be done by looking a the smallest real part of the eigenvalues of the matrix $A - BB^* \Sigma$. To evaluate the closeness to our desired target distribution, we compare the empirical covariance $\hat{\Sigma}_{rr,t}$ to the desired covariance $\Sigma_{rr}$ by means of the scaled Frobenius norm $\frac{1}{d}\|\hat{\Sigma}_{rr,t} -  \Sigma_{rr}\|_F$. Figure \ref{chain_of_oscillator_BT_and_Pavon} displays this measure as a function of time by simulating $k$ different realizations of the reduced controlled process up to $T=30$. We see that we indeed get very close to the desired target, in particular if we choose $k$ large enough. The time discretization of the Euler-Maruyama scheme that we use for discretization seems to be small enough in all trials. 

\begin{figure}[H]
\centering
  \includegraphics[width=0.5\linewidth]{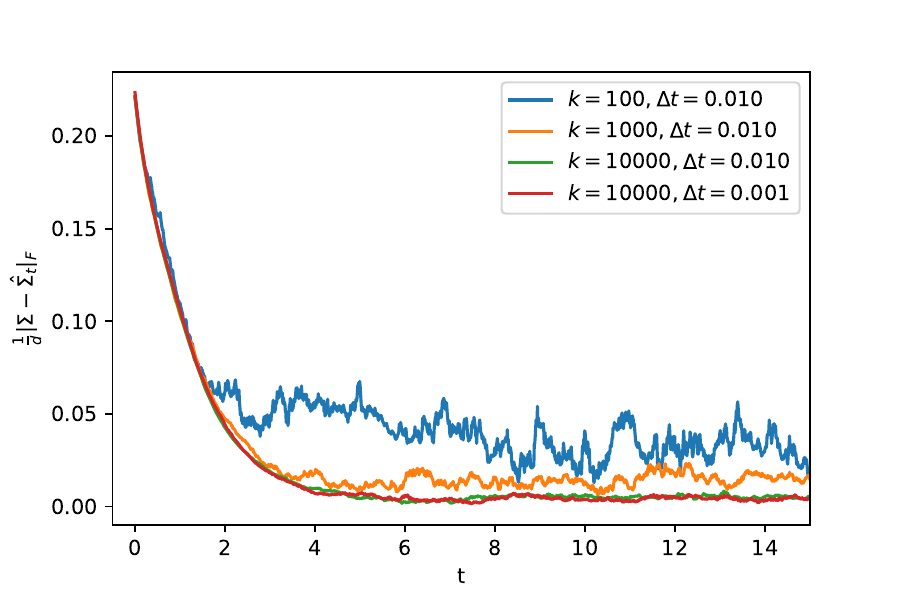}
  \caption{Convergence of the reduced chain of oscillator system to the desired target distribution.}
  \label{chain_of_oscillator_BT_and_Pavon}
\end{figure}
\end{ex}
\smallsection{Acknowledgements} 
This work was partly supported by the EPSRC grant EP/L016516/1 for the University of Cambridge CDT, the CCA (S.B.), and the DFG Collaborative Research Center 1114 ``Scaling Cascades in Complex Systems'', project A05 (L.R.). 

%



\begin{thebibliography}{0}
\bibitem[BS91]{BS91} Barles, G. and Souganidis, P.E. (1991). \emph{Convergence of approximation schemes for fully nonlinear second order equations}, Asymptot. Anal. 4, 271-283.
\bibitem[BGM17]{BGM} Beattie, C.A., Gugercin, S., and Mehrmannn, V. (2017). \emph{Model reduction for systems with inhomogeneous initial conditions}. Systems \& Control Letters 99:99-106.
\bibitem[BH19]{BH19} Becker, S. and Hartmann, C. (2019). \emph{Infinite-dimensional bilinear-and stochastic balanced truncation with error bounds}, Mathematics of Control, Signals, and Systems. 31:5.
\bibitem[BM19]{BM18} Becker, S. and Menegaki, A. (2019). \emph{Spectral gap in $\mathcal O(n)$-model and chain of oscillators using Schr\"odinger operators}, \arXiv{1909.12241}.
\bibitem[BD14]{BD3} Benner, P.  and Damm, T. (2014). \emph{Balanced Truncation for Stochastic Linear Systems with Guaranteed Error Bound}. Proceedings of MTNS-2014. 1492-1497. 
\bibitem[BD11]{BD11} Benner, P. and Damm, T. (2011). \emph{Lyapunov Equations, Energy Functionals, and Model Order Reduction of Bilinear and Stochastic Systems}, SIAM J. Control Optim. 49(2), SIAM J. Control Optim., 49(2), 686–711.
\bibitem[B16]{B1} Benner, B., Ohlberger, M., Patera, T., Rozza, G., and Urban, K. (2016). \emph{Model Reduction of Parametrized Systems- Modeling, Simulation and Applications}. Springer International Publishing, Cham.
\bibitem[BR15]{BR15} Benner, P. and Redmann, M. (2015). \emph{Model reduction for stochastic systems.} Stoch PDE: Anal Comp 3: 291.
\bibitem[BLR00]{BLR00} Bonetto, F.,  Lebowitz, J. L., and Rey-Bellet, L.(2000). \emph{Fourier\'{s} law: a challenge to theorists.} In: Mathematical physics 2000 (London: Imp. Coll. Press, 2000), pp. 128--150.
\bibitem[ABG12]{BorkarBook} Arapostathis, A., Borkar, V. S., and Ghosh, M. K. (2012). \emph{Ergodic Control of Diffusion Processes}, Cambridge Univ. Press, 2012.
\bibitem[BD10]{BD2} Breiten, T. and Damm, T. (2010). \emph{Krylov subspace methods for model order reduction of bilinear control systems}. Systems and Control Letters, Elsevier.
\bibitem[BMS21]{LQGPHS} Breiten, T., Morandin, R., and Schulze, P. (2021). \emph{Error bounds for port-Hamiltonian model and controller reduction based on system balancing}, Comput. Math. Appl., in print. 
\bibitem[B65]{Bucy1965} Bucy, R.S. (1965). \emph{Nonlinear filtering theory.} IEEE Transactions on Automatic Control 10, 198--198. 
\bibitem[C85]{Casti85}
Casti, J.L. (1985). \emph{Nonlinear System Theory.}  Academic Press, Inc., Orlando, Florida. 
\bibitem[CG86]{CG} Curtain, R. and Glover K. (1986). \emph{Balanced realisation for infinite-dimensional systems. Operator Theory and Systems}. Birkh\"auser, Boston, MA.
\bibitem[Cu03]{LQGPS} Curtain, R.F. (2003). \emph{Model reduction for control design for distributed parameter systems}, in: Smith, R. and Demetriou, M. (Eds.), Research Directions in Distributed Parameter Systems, SIAM, Philadelphia, PA, USA, pp. 95-121.
\bibitem[CGP88]{CGP} Glover, K., Curtain, R., and Partington, J. (1988). \emph{Realisation and Approximation of Linear Infinite-Dimensional Systems with Error Bounds}. SIAM Journal on Control and Optimization 26:4, 863-898.
\bibitem[CGP16]{CGP2} Chen, Y., Georgiou, T., and Pavon, M. (2016). \emph{Optimal Steering of a Linear Stochastic System to a Final Probability Distribution, Part II}. IEEE Transactions on Automatic Control, Volume: 61, 5.
\bibitem[DHQ19]{DHQ} Daraghmeh, A., Hartmann, C., Qatanani, N. (2019). \emph{Balanced model reduction of linear systems with nonzero initial conditions: Singular perturbation approximation}.
Appl. Math. Comput. 353, 295--307.
\bibitem[DFV14]{Donev2014}
Donev, A., Fai, T.G., and Vanden-Eijnden, E. (2019). \emph{A reversible mesoscopic model of diffusion in liquids: from giant fluctuations to {F}ick's law}. J. Stat. Mech. Theor. Exp. 2014(4), P04004. 
\bibitem[EN00]{EN} Engel, K-J. and Nagel, R. (2000). \emph{One-Parameter Semigroups for Linear Evolution Equations}. Springer. Graduate Texts in Mathematics.
\bibitem[ES19]{ES} Endres, S. and St\"ubinger, J. (2019). \emph{Optimal trading strategies for L{\'e}vy-driven Ornstein-Uhlenbeck processes}. Applied Economics, 2019 - Taylor \& Francis.
\bibitem[FS06]{FlemingSoner} Fleming, W. H., Soner, H. M. (2006). \emph{Controlled Markov Processes and Viscosity Solutions}, Springer-Verlag, New York. 
\bibitem[FR18]{FR2} Freitag, M.  and Redmann, M. (2018). \emph{Balanced model order reduction for linear random dynamical systems driven by L\'evy noise.}  Journal of Computational Dynamics, 5, 33-59, 27.
\bibitem[G84]{Glover84}
Glover, K. (1984). \emph{All optimal Hankel-norm approximations of linear multivariable systems and their $L^{\infty}$-error bounds.} Intl. J. Control 39, 1115--1193.  
\bibitem[GA04]{BTsurvey} Gugercin, S. and Antoulas, A.C. (2004). \emph{A Survey of Model Reduction by Balanced Truncation and Some New Results}, Intl. J. Control, 77(8), 748-766.
\bibitem[GAB08]{H2survey} Gugercin, S., Antoulas, A. C., and Beattie, C. A. (2008) \emph{H2 model reduction for large-scale dynamical systems}, SIAM J. Matrix Anal. Appl. 30, 609–638.
\bibitem[HLPZ14]{HLZ}
Hartmann, C., Latorre, J. C., Zhang, W., and Pavliotis, G. A. (2014). \emph{Optimal control of multiscale systems using reduced-order models}, J. Computational Dynamics 1, 279-306. 
\bibitem[HNS21]{HNS} Hartmann, C., Neureither, L., and Strehlau, M. (2021). \emph{Reachability Analysis of Randomly Perturbed Hamiltonian Systems}, 7th IFAC Workshop on Lagrangian and Hamiltonian Methods for Nonlinear Control (LHMNC21), in print.\bibitem[H09]{H} Hull, J. (2009). \emph{Options, Futures, and other Derivatives} (7 ed.).
\bibitem[HRA11]{HRA} Heinkenschloss, M., Reis, T., and Antoulas, A.C. (2011).\emph{Balanced truncation model reduction for systems with inhomogeneous initial conditions.} Automatica
\bibitem[JS83]{LQGBT} Jonckheere, E. and Silverman, L. (1983). \emph{A new set of invariants for linear systems---Application to reduced order compensator design}, IEEE Trans. Aut. Control 28(10), 953-964.
\bibitem[K07]{K} Van Kampen, N.G. (2007). \emph{Stochastic Processes in Physics and Chemistry}. 3. Auflage. North Holland.
\bibitem[KNH18]{KNH18} Kebiri, O., Neureither, L., and Hartmann, C. (2018). \emph{Singularly perturbed forward-backward stochastic differential equations: application to the optimal control of bilinear systems}, Computation 6(3), 41-59. 
\bibitem[Kl68]{Kleinman} Kleinman, D.L. (1968). \emph{On an iterative technique for Riccati equation computations}. IEEE Trans. Automatic Control ACo13, 114-115.
\bibitem[KV08]{POD} Kunisch, K. und Volkwein, S. (2008). \emph{Proper orthogonal decomposition for optimality systems}, ESAIM M2AN 42, 1-23.
\bibitem[LLR67]{LLR} Lieb, E., Lebowitz, J.L., and Rieder, Z. (1967). \emph{Properties of a Harmonic Crystal in a Stationary Nonequilibrium State.} In: Nachtergaele B., Solovej J.P., Yngvason J. (eds) Statistical Mechanics. Springer, Berlin, Heidelberg.
\bibitem[MHKZ89]{Medina1989}
Medina, E., Hwa, T.,  Kardar, M., and Zhang, Y.-C. (1989). \emph{Burgers equation with correlated noise: Renormalization-group analysis and applications to directed polymers and interface growth.} Phys.Rev. A 39,  3053--3075. 
\bibitem[M19]{M19}Menegaki, A. (2019). \emph{Quantitative Rates of Convergence to Non-Equilibrium Steady State for a Weakly Anharmonic Chain of Oscillators}. \arXiv{1909.11718}
\bibitem[NR21]{NR21} N{\"u}sken, N. and Richter, L. (2021). \emph{Solving high-dimensional Hamilton--Jacobi--Bellman PDEs using neural networks: perspectives from the theory of controlled diffusions and measures on path space.} In: Partial Differential Equations and Applications 2, 4, Springer, Berlin, Heidelberg.
\bibitem[OC05]{OC05} Opmeer, M. R. and Curtain, R. F. (2005). \emph{New Riccati equations for well-posed linear systems}, 
Syst. Control Lett. 52(5), 339-347.
\bibitem[ORW13]{ORW}Opmeer, M., Reis, T., and Wollner, W. (2013). \emph{Finite-Rank ADI Iteration for Operator Lyapunov Equations}. SIAM J. Control and Optimization, 51(5), 4084-4117.
\bibitem[PZ07]{P15} Peszat, S. and Zabczyk, J. (2007). \emph{Stochastic Partial Differential Equations with L\'{e}vy Noise.} Cambridge University Press. 
\bibitem[R18]{redmannspa2} Redmann, M. (2018). \emph{Type II singular perturbation approximation for linear systems with L\'evy noise}. SIAM J. Control and Optimization, 56(3):2120-2158.
\bibitem[RS14]{ReSe} Reis, T. and Selig, T. (2014). \emph{Balancing Transformations for
Infinite-Dimensional Systems with Nuclear Hankel operator.} T. Integr. Equ. Oper. Theory, Volume 79, Issue 1, pp 67-105.
\bibitem[RZ00]{RZ}Rami, M. A. and Zhou, X. Y. (2000). \emph{Linear Matrix Inequalities, Riccati Equations, and Indefinite Stochastic Linear Quadratic Controls.} IEEE Transactions on Automatic Control, Vol. 45, No. 6.
\bibitem[Ro05]{Rowley} Rowley, C. W. (2005). \emph{Model Reduction for Fluids, using Proper Orthogonal Decomposition}, Intl. J. Bifurc. Chaos 15(3), 997-1013.
\bibitem[SDS21]{SdS} B. Salarieh and H.M.J. {De Silva} (2021). \emph{Review and comparison of frequency-domain curve-fitting techniques: Vector fitting, frequency-partitioning fitting, matrix pencil method and loewner matrix},  Electr. Power Syst. Res. 196, 107254. 
\bibitem[SHSS11]{SHSS} Sch\"afer-Bung, B., Hartmann, C., Schmidt, B. and Sch\"utte, C. (2011). \emph{Dimension reduction by balanced truncation: Application to light-induced control of open quantum systems}. J. Chem. Phys. 135, 014112.
\bibitem[SVR08]{SchildersBook} Schilders, W. H. A., van der Vorst, H. A., and Rommes, J. (2008). \emph{Model Order Reduction: Theory, Research Aspects and Applications}, Springer-Verlag, Berlin, Heidelberg. 
\bibitem[V77]{V} Vasicek, O. (1977). \emph{An equilibrium characterization of the term structure}. Journal of Financial Economics. 5 (2): 177-188.
\bibitem[Z75]{Z} Zabczyk, J. (1975). \emph{Remarks on the algebraic Riccati equation in Hilbert space}. Remarks on the algebraic Riccati equation in Hilbert space.
\end{thebibliography}
\end{document}